\documentclass[final,onefignum,onetabnum]{siamart251216}

\usepackage{amsmath,amsfonts,amsopn,amssymb}
\usepackage{graphicx}
\usepackage[caption=false]{subfig}
\usepackage{multirow,relsize,makecell}
\usepackage{url}
\ifpdf

  \DeclareGraphicsExtensions{.eps,.pdf,.png,.jpg}
\else
  \DeclareGraphicsExtensions{.eps}
\fi
\usepackage{enumitem}
\usepackage{accents}
\usepackage{hyperref}
\usepackage{accents}

\usepackage{algorithm,algorithmic}

\numberwithin{theorem}{section}
\newsiamremark{remark}{Remark}
\newsiamremark{assumption}{Assumption}
\newsiamremark{example}{Example}
\crefname{remark}{Remark}{Remarks}
\crefname{assumption}{Assumption}{Assumptions}
\crefname{example}{Example}{Examples}

\headers{Dimension dependence of BPX preconditioners}{Boou Jiang, Jongho Park, and Jinchao Xu}

\title{A polynomial dimension-dependence analysis\\ of Bramble--Pasciak--Xu preconditioners 
\thanks{Submitted to arXiv.\funding{This work was supported by the KAUST Baseline Research Fund.}}
}

\author{
Boou Jiang\thanks{Applied Mathematics and Computational Sciences Program, Computer, Electrical and Mathematical Science and Engineering Division, King Abdullah University of Science and Technology~(KAUST), Thuwal 23955, Saudi Arabia
 (\email{boou.jiang@kaust.edu.sa}, 
 \email{jongho.park@kaust.edu.sa},
 \email{jinchao.xu@kaust.edu.sa}).}
\and
Jongho Park\footnotemark[2]
\and
Jinchao Xu\footnotemark[2]
 }

\ifpdf
\hypersetup{
  pdftitle={A polynomial dimension-dependence analysis of Bramble--Pasciak--Xu preconditioners},
  pdfauthor={Boou Jiang, Jongho Park, and Jinchao Xu}
}
\fi

\usepackage{color}

\begin{document}

\maketitle

\begin{abstract}
We investigate the dimension dependence of Bramble--Pasciak--Xu~(BPX) preconditioners for simplicial $\mathbb{P}_1$ finite element discretizations of high-dimensional partial differential equations and establish that the condition numbers of BPX-preconditioned systems grow only polynomially with the spatial dimension.
Our analysis requires a careful derivation of the dimension dependence of several fundamental tools in the theory of finite element methods, including elliptic regularity, the Bramble--Hilbert lemma, trace inequalities, and inverse inequalities.
We further analyze an averaged Scott--Zhang-type quasi-interpolation operator, and show that its associated constants scale polynomially with the dimension.
Building on these ingredients, we prove a multilevel norm equivalence theorem and derive a BPX preconditioner with explicit polynomial bounds on its dimensional dependence.
The analysis is motivated in part by recent tensor and quantum finite element methods, where dimension-explicit conditioning estimates for BPX preconditioners play an important role.
\end{abstract}



\section{Introduction}
\label{Sec:Introduction}
The main purpose of this paper is to show that the Bramble--Pasciak--Xu (BPX) preconditioners~\cite{BPX:1990,Xu:1989} admit condition numbers whose dependence on the spatial dimension is only polynomial, under suitable geometric assumptions on the underlying triangulations.
While BPX preconditioners and their variants have been widely applied to two- and three-dimensional problems arising in computational mechanics and other applications (see, e.g.,~\cite{BNWX:2023,BHKS:2013,LX:2016,Xu:1996}), there appear to be very few works addressing multilevel preconditioning in high-dimensional regimes.
A multilevel preconditioner for full and sparse tensor-product rectangular grids in higher dimensions was studied in~\cite{GH:2014}.

BPX preconditioners were first proposed in~\cite{BPX:1990,Xu:1989} as a result of efforts to parallelize multilevel iterative methods for the numerical solution of partial differential equations~(PDEs), and they are now regarded as one of the major multigrid approaches for solving large-scale algebraic systems arising from numerical PDEs.
Within the framework of subspace correction methods~\cite{Xu:1992,XZ:2002}, BPX preconditioners can be viewed as parallel subspace correction methods based on a multilevel space decomposition.
BPX preconditioners have also motivated important theoretical developments in multilevel iterative methods, including the multilevel norm equivalence theorem~\cite{BY:1993,DK:1992,Oswald:1990,XQ:1994}, which implies that the condition number of the BPX preconditioner is independent of the mesh size and the number of levels, and has since become a foundation for the design and analysis of many modern multilevel algorithms.
In low-rank tensor methods for elliptic problems, the numerical stability of tensor representations has been analyzed in connection with multilevel preconditioning~\cite{BK:2020}.

The importance of BPX preconditioners has been newly recognized recently in connection with the development of quantum algorithms for solving large-scale linear systems (see, e.g.,~\cite{Ambainis:2012,CKS:2017,HHL:2009}), which have the potential to provide exponential speedup over classical algorithms.
In particular, quantum algorithms have been applied to finite element methods for solving various numerical PDEs; see, e.g.,~\cite{JLY:2025a,JLY:2025b,MP:2016}.

Since the performance of these quantum algorithms depends strongly on the condition number of the target linear system, BPX-type preconditioning
has recently been used in quantum solvers based on tensor-product finite element discretizations~\cite{DP:2026} and finite difference discretizations of the Poisson equation~\cite{JLMY:2025}. 
Motivated by these developments, we carry out a
dimension-explicit spectral analysis of BPX preconditioners.
We note that polynomial condition number bounds alone do not establish an efficient quantum implementation.

In this paper, we prove that the condition number of BPX-preconditioned systems arising from continuous, piecewise linear finite element methods in high dimensions depends on the spatial dimension only polynomially.
To achieve this, we first carefully derive the dimension dependence of several fundamental tools in finite element analysis, including the Bramble--Hilbert lemma~\cite{BH:1970}, trace inequalities, and inverse inequalities~\cite{BS:2008,Ciarlet:2002}, whereas most of the existing literature does not explicitly track this dependence.

We then use an averaged Scott--Zhang-type quasi-interpolation operator, which is a projective quasi-interpolation operator based on local \(L^2\)-projection followed by nodal averaging; see, e.g.,~\cite{EG:2017,KPY:2018}.
We derive \(L^2\)- and \(H^1\)-error estimates for this operator with explicit polynomial dependence on \(d\), \(\rho\), and \(\sigma\), including the treatment of boundary nodes.
We observe that, with conventional quasi-interpolation operators such as the Scott--Zhang interpolation~\cite{SZ:1990}, it is difficult to obtain polynomial dependence directly; see \cref{App:SZ} for an example in which the standard Scott--Zhang interpolation constant grows super-exponentially with \(d\).

Using these ingredients, we derive a multilevel norm equivalence theorem~\cite{BY:1993,DK:1992,Oswald:1990,XQ:1994} and construct a BPX preconditioner involving the exact $L^2$-orthogonal projections whose dependence on the spatial dimension is explicitly polynomial.
Incorporating the theory of parallel subspace correction methods~\cite{PX:2025,XZ:2002} provides an even sharper upper bound for the condition number of the BPX preconditioner.
We also discuss other versions of BPX preconditioners with Jacobi and Richardson smoothers, which are more relevant to recent quantum implementations of BPX-type preconditioning~\cite{DP:2026,JLMY:2025}.

This paper is organized as follows.
In \cref{Sec:Notation}, we introduce the notation used throughout the paper.
In \cref{Sec:Triangulations}, we summarize important notions of triangulations in finite element methods, with emphasis on their dependence on the spatial dimension.
In \cref{Sec:Tools}, we analyze the dimension dependence of several fundamental tools in finite element analysis and multilevel iterative methods.
In \cref{Sec:SZ}, we introduce the averaged Scott--Zhang quasi-interpolation and establish its error estimates.
In \cref{Sec:BPX}, we derive the multilevel norm equivalence theorem and several versions of BPX preconditioners, and we analyze their dependence on the dimension.
Finally, in \cref{Sec:Conclusion}, we conclude the paper with remarks.

\section{Notation}
\label{Sec:Notation}
This section is devoted to the notation used throughout the paper.

Following standard conventions in the multilevel method literature (see, e.g.,~\cite{Xu:1992}),  
we write \(x \lesssim y\) (equivalently \(y \gtrsim x\)) if there exists a constant \(C > 0\), 
independent of all important parameters (in this paper, these include the mesh size, number of levels, 
dimension, shape-regularity, quasi-uniformity), such that \(x \le C y\).
We write \(x \eqsim y\) when both \(x \lesssim y\) and \(y \lesssim x\) hold. 

Regarding matrices and vectors, we use the following notation:
\begin{itemize}
    \item \(\mathsf{I}_d\): the \(d \times d\) identity matrix;
    \item \(\mathsf{0}_d\): the \(d\)-dimensional vector whose entries are all equal to \(0\).
    \item \(\mathsf{1}_d\): the \(d\)-dimensional vector whose entries are all equal to \(1\).
\end{itemize}
For finite element functions, we identify functions with their corresponding vectors of degrees of freedom whenever convenient.

\subsection{Simplices and triangulations}
Next we introduce notation related to simplices.
Throughout the paper, all simplices are regarded as open sets.  
For a \(d\)-simplex \(\tau \subset \mathbb{R}^d\), we write:
\begin{itemize}
    \item \(h_\tau\): the diameter of \(\tau\);
    \item \(r_\tau\): the inradius of \(\tau\);
    \item \(\mathcal{V}(\tau)\): the set of vertices of \(\tau\);
    \item \(\mathcal{F}(\tau)\): the set of \((d-1)\)-dimensional faces of \(\tau\).
\end{itemize}
Note that
\[
\# (\mathcal{V}(\tau) ) = \# (\mathcal{F}(\tau) ) = d+1.
\]

Given a bounded polyhedral domain \(\Omega \subset \mathbb{R}^d\), let \( \mathcal{T}_h \) be a conforming triangulation of \(\Omega\) with characteristic mesh size \(h\).
That is, \( \mathcal{T}_h \) is a collection of disjoint \(d\)-simplices (elements) such that
\[
\overline{\Omega} = \bigcup_{\tau \in \mathcal{T}_h} \overline{\tau},
\quad
h = \max_{\tau \in \mathcal{T}_h} h_\tau.
\]

For the triangulation \(\mathcal{T}_h\), we use the following notation:
\begin{itemize}
    \item $\mathcal{T}_h (\partial \Omega)$: the set of all elements that have at least one \((d-1)\)-face on \(\partial \Omega\);
    \item \(\mathcal{V}_h\): the set of all vertices of \(\mathcal{T}_h\);
    \item \(\mathcal{V}_h(\partial\Omega)\): the set of all vertices of \(\mathcal{T}_h\) lying on \(\partial\Omega\);
    \item \(\mathcal{F}_h\): the set of all \((d-1)\)-faces of \(\mathcal{T}_h\);
    \item \(\mathcal{F}_h(\partial\Omega)\): the set of all \((d-1)\)-faces of \(\mathcal{T}_h\) lying on \(\partial\Omega\).
\end{itemize}

For a vertex \(a \in \mathcal{V}_h\), let \(\omega_a\) denote the vertex patch around $a$, i.e.,
\begin{equation}
\label{vertex_patch}
\overline{\omega}_a
=
\bigcup_{\tau \in \mathcal{T}_h,\ a \in \mathcal{V}(\tau)}
\overline{\tau}.
\end{equation}
Similarly, for an element \(\tau \in \mathcal{T}_h\), the element patch around $\tau$ is denoted by $\omega_{\tau}$:
\begin{equation}
\label{element_patch}
\overline{\omega}_\tau
=
\bigcup_{\tau' \in \mathcal{T}_h,\ \overline{\tau}\cap \overline{\tau}' \neq \emptyset}
\overline{\tau}'.
\end{equation}

\subsection{Finite element spaces}
Given a triangulation $\mathcal{T}_h$ of a bounded domain $\Omega \subset \mathbb{R}^d$, we denote by $V_h$ the lowest-order Lagrangian finite element space, namely the space of continuous, piecewise linear functions on $\mathcal{T}_h$:
\[
    V_h = \{ v \in C(\overline{\Omega}) : v|_{\tau} \in \mathbb{P}_1(\tau) \text{ for all } \tau \in \mathcal{T}_h \}.
\]
To incorporate the Dirichlet boundary condition, we also define
\[
    V_{h,0} = \{ v \in V_h : v = 0 \text{ on } \partial\Omega \}.
\]

For each $a \in \mathcal{V}_h$, there exists a nodal basis function $\phi_a \in V_h$ satisfying
\begin{equation}
\label{phi}
    \phi_a(b) = \delta_{ab}, \quad b \in \mathcal{V}_h,
\end{equation}
where $\delta_{ab}$ denotes the Kronecker delta.

\section{Triangulations in high dimensions}
\label{Sec:Triangulations}
In this section, we review important notions of triangulations in finite element methods, such as shape-regularity and quasi-uniformity, and discuss their dependence on the spatial dimension.
In addition, we present the Freudenthal triangulation as a concrete example of a triangulation in high dimensions and examine its properties.

Throughout this section, we assume that $\mathcal{T}_h$ is a triangulation of a bounded polyhedral domain $\Omega \subset \mathbb{R}^d$.

\subsection{Shape-regularity and quasi-uniformity}
Shape-regularity and quasi-uniformity of triangulations are standard notions that appear throughout the finite element literature; see~\cite{BS:2008,Ciarlet:2002} for standard references.
Here, we review these concepts and highlight how the high-dimensional setting affects them.

We begin with the definition of shape-regularity.

\begin{definition}[shape-regularity]
\label{Def:shape-regular}
A triangulation $\mathcal{T}_h$ is said to be $\rho$-shape-regular if there exists a constant $\rho > 0$ such that
\[
   h_\tau \le \rho\, r_\tau, \quad \tau \in \mathcal{T}_h.
\]
\end{definition}

A useful property of shape-regular triangulations is that the height of each element is uniformly comparable to its diameter, as summarized in \cref{Lem:shape-regular}.

\begin{lemma}
\label{Lem:shape-regular}
Let $\mathcal{T}_h$ be a $\rho$-shape-regular triangulation of $\Omega \subset \mathbb{R}^d$, and let $\tau \in \mathcal{T}_h$.
For a given $(d-1)$-face $F \in \mathcal{F}(\tau)$, let $h_F$ denote the distance from the vertex of $\tau$ opposite to $F$ to the hyperplane containing $F$.
Then we have
\[
    \rho^{-1} h_\tau \lesssim h_F \le h_\tau.
\]
\end{lemma}

\begin{proof}
The inequality $h_F \le h_\tau$ is immediate, so it remains to prove $h_F \gtrsim \rho^{-1} h_\tau$.
Computing the volume of $\tau$ in two different ways gives
\begin{equation}
\label{Lem1:shape-regular}
    h_F |F| = r_\tau \sum_{F' \in \mathcal{F}(\tau)} |F'|.
\end{equation}
Hence, we get
\[
    h_F
    = r_\tau \cdot \frac{1}{|F|} \sum_{F' \in \mathcal{F}(\tau)} |F'|
    \ge 2 r_\tau
    \ge 2 \rho^{-1} h_\tau,
\]
where the first inequality follows from the simplex inequality~\cite{Izumi:2016}, and the second follows from \cref{Def:shape-regular}.
\end{proof}

One important observation is that the shape-regularity parameter $\rho$ implicitly depends on the dimension $d$.
Therefore, in what follows, unlike much of the finite element literature where the dependence on the shape-regularity parameter is omitted for simplicity, we will explicitly keep track of the dependence on $\rho$ as well.
Indeed, we have the following lower bound for the shape ratio in terms of the dimension $d$.

\begin{proposition}
\label{Prop:shape-regular}
For a $d$-simplex $\tau \subset \mathbb{R}^d$, we have
\begin{equation*}
    \frac{h_{\tau}}{r_{\tau}} \gtrsim d.
\end{equation*}
\end{proposition}
\begin{proof}
For each $(d-1)$-face $F \in \mathcal{F}(\tau)$, we define $h_F$ as in \cref{Lem:shape-regular}.
Then 
\begin{equation*}
r_{\tau}^{-1}
\stackrel{\eqref{Lem1:shape-regular}}{=} \sum_{F \in \mathcal{F}(\tau)} h_F^{-1}
\ge (d+1) h_{\tau}^{-1},
\end{equation*}
where the last inequality follows from \(h_F\le h_\tau\) for every
\(F\in\mathcal F(\tau)\).
This completes the proof.
\end{proof}

\begin{remark}
\label{Rem:shape-regular}
The estimate in \cref{Prop:shape-regular} is sharp.
Indeed, the ratio between the diameter and the inradius of a regular $d$-simplex is given by $\sqrt{2d(d+1)} \eqsim d$.
\end{remark}

While shape-regularity measures how similar the shapes of elements in a given triangulation are, we also need a notion that measures how similar their sizes are.
This notion, called quasi-uniformity, is given in \cref{Def:quasi-uniform}.

\begin{definition}[quasi-uniformity]
\label{Def:quasi-uniform}
A triangulation $\mathcal{T}_h$ is said to be $(\rho,\sigma)$-quasi-uniform if it is $\rho$-shape-regular and there exists a constant $\sigma > 0$ such that
\[
   \sigma h \le h_\tau \le h , \quad \tau \in \mathcal{T}_h.
\]
\end{definition}

Unlike \(\rho\), \(\sigma\) need not grow with \(d\) for uniform triangulations.
Nevertheless, for completeness, we will also keep track of the dependence on the quasi-uniformity parameter $\sigma$ in the analysis presented in this paper.

\subsection{Freudenthal triangulations}
Constructing simplicial meshes in high dimensions is highly nontrivial.
Here we present the Freudenthal triangulation~\cite{Bey:2000} as an example of a structured simplicial mesh in arbitrary spatial dimensions.
This triangulation has been used in the literature on structure-preserving finite element methods in high dimensions; see, e.g.,~\cite{FMS:2024,Zhang:2011}.

For simplicity, we describe the Freudenthal triangulation of the unit cube \((0,1)^d \subset \mathbb{R}^d\).
For each permutation \(s \colon \{1,\dots,d\} \to \{1,\dots,d\}\), we define
\[
    \tau_s
    =
    \{
        x = (x_1,\dots,x_d) \in (0,1)^d :
        0 < x_{s(1)} < x_{s(2)} < \dots < x_{s(d)} < 1
    \}.
\]
The vertices of \(\tau_s\) are given by
\[
    \mathsf{0}_d, \quad
    \mathsf{e}_{s(d)}, \quad
    \mathsf{e}_{s(d)} + \mathsf{e}_{s(d-1)}, \quad
    \dots, \quad
    \mathsf{e}_{s(d)} + \dots + \mathsf{e}_{s(1)} = \mathsf{1}_d,
\]
where \(\mathsf{e}_i\) is the \(i\)th canonical basis vector of \(\mathbb{R}^d\).
The Freudenthal triangulation of the unit cube is the collection of all such simplices \(\tau_s\), one for each permutation \(s\).

One can readily observe that all \(d\)-simplices in the Freudenthal triangulation are congruent and satisfy
\[
    |\tau_s| = \frac{1}{d!}, \quad
    h_{\tau_s} = \sqrt{d}, \quad
    r_{\tau_s} = \frac{1}{2 + (d-1)\sqrt{2}}.
\]
Hence, the ratio between the diameter and the inradius is
\[
    \frac{h_{\tau_s}}{r_{\tau_s}}
    = 2\sqrt{d} + (d-1)\sqrt{2d}
    \eqsim d^{3/2},
\]
which shows that the shape-regularity parameter \(\rho\) of the Freudenthal triangulation satisfies
\(\rho \eqsim d^{3/2}\), and therefore is not optimal in the sense of \cref{Prop:shape-regular}.
This observation suggests that maintaining high-quality simplicial meshes becomes increasingly challenging in high dimensions.

The same construction applies cellwise to domains represented as unions of cells in a Cartesian cubical partition: each cube is mapped affinely to \((0,1)^d\) and subdivided by the above Freudenthal rule.
Moreover, successive uniform refinements of the cubical mesh, followed by the same cellwise Freudenthal subdivision, give a nested simplicial hierarchy.

For such structured Freudenthal meshes, every interior vertex patch defined in~\eqref{vertex_patch} is convex, since it can be characterized as the intersection of half-spaces. 
This property motivates the convex vertex-patch assumption used throughout the paper; see \cref{Ass:triangulation}. 
On the other hand, one can observe that an element patch defined in~\eqref{element_patch} is nonconvex in general when \(d \ge 3\).

\begin{remark}
\label{Rem:FDM}
It is known that the stiffness matrix of the standard finite difference
discretization of the Dirichlet Laplacian on a Cartesian grid coincides, up to
a multiplicative factor, with that of the \(\mathbb{P}_1\) finite element
discretization on the corresponding uniform Freudenthal triangulation;
see~\cite[Equation~(3.3)]{Bachmayr:2023}.
Hence, the finite difference method on structured grids is covered by the
Freudenthal finite element setting considered in this paper.
\end{remark}

\section{Fundamental tools}
\label{Sec:Tools}
In this section, we revisit several standard inequalities in finite element theory while specifying their dependence on the dimension \(d\), the shape-regularity parameter \(\rho\), and the quasi-uniformity parameter \(\sigma\).
The estimates themselves are classical; the purpose of this section is to record them in the forms needed later and to keep explicit the constants that affect the dimension-dependence analysis.



\subsection{Regularity on convex domains}
Elliptic regularity plays a central role in the error analysis of finite element methods for elliptic problems~\cite{BS:2008,Ciarlet:2002}.
For convex domains, it is known, often referred to as the Miranda--Talenti estimate~\cite[Lemma~1.2.2]{MPS:2000} (see also~\cite[Theorem~3.1.2.1]{Grisvard:2011}), that the elliptic regularity constant for bounding the $| \cdot |_{H^2 (\Omega)}$-seminorm is $1$, and thus is independent of the dimension, the shape of the domain, and related factors.
For completeness, we state this fact in \cref{Lem:regularity}.

\begin{lemma}[elliptic regularity on convex domains]
\label{Lem:regularity}
Let $\Omega \subset \mathbb{R}^d$ be a convex, bounded, and piecewise smooth domain.
Given $f \in L^2 (\Omega)$, let $u \in H_0^1 (\Omega)$ be the solution of the variational problem
\begin{equation*}
    \int_{\Omega} \nabla u \cdot \nabla v \,dx = \int_{\Omega} f v\,dx,
    \quad \quad v \in H_0^1 (\Omega).
\end{equation*}
Then we have $u \in H_0^1 (\Omega) \cap H^2 (\Omega)$ and
\begin{equation*}
    |u |_{H^2 (\Omega)} \leq \| f \|_{L^2 (\Omega)}.
\end{equation*}
\end{lemma}

\subsection{Bramble--Hilbert lemma}
The Bramble--Hilbert lemma~\cite{BH:1970} is a fundamental tool in polynomial approximation theory in Sobolev spaces~\cite[Chapter~4]{BS:2008}.
Here we use a version on convex domains whose constant is independent of the dimension, the shape of the domain, and other related geometric factors.
The estimate is closely related to the convex-domain Bramble--Hilbert estimates in~\cite{DL:2004}; we include the short proof to make the dimension dependence explicit.

We begin by recalling the sharp Poincar\'e inequality for convex domains, as presented in~\cite{Bebendorf:2003,PW:1960}.
This result is particularly noteworthy, since for nonconvex domains---even those with favorable geometric structure such as vertex patches~\cite{VV:2012}---the situation becomes substantially more complicated, and the dependence on the dimension can grow as large as $\mathcal{O}(d^d)$, which is prohibitively severe.

\begin{lemma}[Poincar\'e inequality on convex domains]
\label{Lem:Poincare}
Let $\Omega \subset \mathbb{R}^d$ be a bounded convex domain.
Then we have
\begin{equation*}
    \left\| v - \frac{1}{| \Omega |} \int_{\Omega} v \,dx \right\|_{L^2 (\Omega)} \leq \frac{\operatorname{diam} (\Omega)}{\pi}| v |_{H^1 (\Omega)}, 
    \quad v \in H^1(\Omega).
\end{equation*}
\end{lemma}

We shall also use the following Friedrichs inequality~\cite{Protter:1981}, which can be viewed as the Dirichlet counterpart of the Poincar\'e inequality in~\cref{Lem:Poincare}. 
The form stated below applies to functions satisfying homogeneous Dirichlet boundary conditions.

\begin{lemma}[Friedrichs inequality on convex domains]
\label{Lem:Friedrichs}
Let $\Omega \subset \mathbb{R}^d$ be a bounded convex domain.
Then we have
\begin{equation*}
    \left\| v \right\|_{L^2 (\Omega)} \leq \frac{\operatorname{diam} (\Omega)}{\pi}| v |_{H^1 (\Omega)}, 
    \quad v \in H_0^1(\Omega).
\end{equation*}
\end{lemma}

Using \cref{Lem:Poincare}, we derive in \cref{Lem:BH} a sharp estimate for the Bramble--Hilbert lemma on convex domains.
We note that the Bramble--Hilbert lemma for convex domains in more general settings, involving broader classes of Sobolev spaces, was studied in~\cite{DL:2004}.

\begin{lemma}[Bramble--Hilbert lemma on convex domains]
\label{Lem:BH}
Let $\Omega \subset \mathbb{R}^d$ be a bounded convex domain.
Then we have
\begin{equation*}
    \inf_{p \in \mathbb{P}_1 (\Omega)} | v - p |_{H^j (\Omega)} \lesssim \operatorname{diam} (\Omega)^{2-j} | v |_{H^2 (\Omega)}, \quad v \in H^2(\Omega),\ j = 0, 1.
\end{equation*}
\end{lemma}
\begin{proof}
Take any $v \in H^2 (\Omega)$, and define $p \in \mathbb{P}_1 (\Omega)$ by
\begin{equation*}
    p (x) = \frac{1}{| \Omega |} \int_{\Omega} v (y) \,dy + \frac{1}{| \Omega | } \int_{\Omega} \nabla v (y) \,dy \cdot \left( x - \frac{1}{| \Omega |} \int_{\Omega} y \,dy \right),
    \quad x \in \Omega.
\end{equation*}
By direct calculation, we get
\begin{equation}
\label{Lem1:BH}
    \int_{\Omega} (v - p)\,dx = 0, \quad
    \int_{\Omega} \nabla (v - p) \,dx = \mathsf{0}_d.
\end{equation}
Hence, by the Poincar\'e inequality~(\cref{Lem:Poincare}), we have
\begin{equation}
\label{Lem2:BH}
    \| v - p \|_{L^2(\Omega)}^2
    \lesssim \operatorname{diam}(\Omega)^2 | v - p |_{H^1 (\Omega)}^2.
\end{equation}
For each $i = 1, \dots, d$,~\eqref{Lem1:BH} implies
\begin{equation*}
    \int_{\Omega} \partial_i (v - p) \,dx = 0,
\end{equation*}
so another application of the Poincar\'e inequality gives
\begin{equation*}
    \| \partial_i (v - p) \|_{L^2(\Omega)}^2
    \lesssim \operatorname{diam}(\Omega)^2 \| \nabla \partial_i(v - p) \|_{L^2(\Omega)}^2.
\end{equation*}
Summing over $i$ yields
\begin{multline}
\label{Lem3:BH}
    | v - p |_{H^1 (\Omega)}^2
    = \sum_{i=1}^d \| \partial_i (v - p) \|_{L^2(\Omega)}^2 \\
    \lesssim \operatorname{diam}(\Omega)^2 \sum_{i=1}^d \| \nabla \partial_i(v - p) \|_{L^2(\Omega)}^2
    \lesssim \operatorname{diam}(\Omega)^2 | v - p |_{H^2 (\Omega)}^2
    = \operatorname{diam}(\Omega)^2 | v |_{H^2 (\Omega)}^2.
\end{multline}
Combining~\eqref{Lem2:BH} and~\eqref{Lem3:BH}, we obtain
\begin{equation*}
    \| v - p \|_{L^2(\Omega)}^2
    \lesssim \operatorname{diam} (\Omega)^2 | v - p |_{H^1 (\Omega)}^2
    \lesssim \operatorname{diam} (\Omega)^4 | v |_{H^2 (\Omega)}^2,
\end{equation*}
which completes the proof.
\end{proof}

\subsection{Trace inequalities}
Trace inequalities are standard tools that relate functions in Sobolev spaces to their traces on the boundary.
Dimension-explicit trace estimates in arbitrary dimension are also available in the literature; see, e.g., \cite[Lemma~7.2]{Gallistl:2023}.
The estimates below are written in the particular forms needed in our analysis, with explicit dependence on \(d\), \(\rho\), and \(\sigma\).
We begin with the trace inequality that relates a shape-regular \(d\)-simplex and one of its \((d-1)\)-faces, stated in \cref{Lem:trace}.

\begin{lemma}[trace inequality on simplices]
\label{Lem:trace}
Let $\mathcal{T}_h$ be a $\rho$-shape-regular triangulation of $\Omega \subset \mathbb{R}^d$.
Then for any $\tau \in \mathcal{T}_h$ and $F \in \mathcal{F} (\tau)$, we have
\begin{equation*}
    \| v \|_{L^2 (F)} \lesssim
    \rho^{\frac{1}{2}} d^{\frac{1}{2}} h_{\tau}^{-\frac{1}{2}} \| v \|_{L^2 (\tau)}
    + \rho^{\frac{1}{2}}d^{-\frac{1}{2}} h_{\tau}^{\frac{1}{2}} | v |_{H^1 (\tau)},
    \quad v \in H^1 (\tau).
\end{equation*}
\end{lemma}
\begin{proof}
It suffices to prove the estimate for $v \in C^1(\overline{\tau})$; the general case follows by density.
Let $x_F$ be the vertex of $\tau$ opposite to $F$.
For $x \in F$, define
\begin{equation*}
    v_x(t) = v(x_F + t(x - x_F)), \quad t \in [0,1].
\end{equation*}
Then, by a change of variables and \cref{Lem:shape-regular},
\begin{equation}
\label{Lem1:trace}
    \int_F \int_0^1 |v_x(t)|^2 t^{d-1}\,dt\,dS(x)
    =
    h_F^{-1}\|v\|_{L^2(\tau)}^2
    \lesssim \rho h_{\tau}^{-1}\|v\|_{L^2(\tau)}^2,
\end{equation}
where $h_F$ denotes the distance from $x_F$ to the hyperplane containing $F$.

On the other hand, integration by parts gives
\begin{equation*}
    v(x)
    =
    d\int_0^1 v_x(t)t^{d-1}\,dt
    +
    \int_0^1 t^d v_x'(t)\,dt .
\end{equation*}
Hence, by the Cauchy--Schwarz inequality,
\begin{equation}
\label{Lem2:trace}
    |v(x)|^2
    \lesssim
    d\int_0^1 |v_x(t)|^2t^{d-1}\,dt
    +
    d^{-1}\int_0^1 |v_x'(t)|^2t^{d-1}\,dt .
\end{equation}
Since
\[
    v_x'(t)
    =
    \nabla v(x_F+t(x-x_F))\cdot(x-x_F),
\]
integrating both sides of~\eqref{Lem2:trace} over $F$ and using $|x-x_F|\le h_\tau$ yields
\begin{equation*}
\begin{split}
    \|v\|_{L^2(F)}^2
    &\lesssim
    d\int_F\int_0^1 |v_x(t)|^2t^{d-1}\,dt\,dS(x) \\
    &\quad
    +
    d^{-1}\int_F\int_0^1
    |\nabla v(x_F+t(x-x_F))\cdot(x-x_F)|^2
    t^{d-1}\,dt\,dS(x) \\
    &\stackrel{\eqref{Lem1:trace}}{\lesssim}
    \rho d h_\tau^{-1}\|v\|_{L^2(\tau)}^2
    +
    \rho d^{-1}h_\tau |v|_{H^1(\tau)}^2.
\end{split}
\end{equation*}
This completes the proof.
\end{proof}

In \cref{Lem:trace}, since $\# ( \mathcal{F} (\tau) ) = d+1$, summing the estimate in \cref{Lem:trace} over all faces yields the following corollary.

\begin{corollary}[trace inequality on simplices]
\label{Cor:trace}
Let $\mathcal{T}_h$ be a $\rho$-shape-regular triangulation of $\Omega \subset \mathbb{R}^d$.
Then for any $\tau \in \mathcal{T}_h$, we have
\begin{equation*}
    \| v \|_{L^2 (\partial \tau)}
    \lesssim
    \rho^{\frac{1}{2}}d h_{\tau}^{-\frac{1}{2}} \| v \|_{L^2 (\tau)}
    +
    \rho^{\frac{1}{2}} h_{\tau}^{\frac{1}{2}} | v |_{H^1 (\tau)},
    \quad v \in H^1 (\tau).
\end{equation*}
\end{corollary}

Next, we consider the discrete trace inequality.
Namely, for \(v \in V_h\), we want to estimate \(\| v \|_{L^2(\partial \Omega)}\) in terms of \(\| v \|_{L^2(\Omega)}\), where $\Omega$ is a bounded domain in $\mathbb{R}^d$.
This can be done directly using the ingredients developed so far:
\begin{equation*}
\begin{split}
    \| v \|_{L^2(\partial \Omega)}^2
    &\leq \sum_{\tau \in \mathcal{T}_h(\partial \Omega)} \| v \|_{L^2(\partial \tau)}^2 \\
    &\lesssim \sum_{\tau \in \mathcal{T}_h(\partial \Omega)}
            \left( \rho\, d^2\, h_\tau^{-1} \| v \|_{L^2(\tau)}^2
            + \rho\, h_\tau\, |v|_{H^1(\tau)}^2 \right) \\
    &\lesssim \sum_{\tau \in \mathcal{T}_h(\partial \Omega)}
            \rho^3 d^3 h_\tau^{-1} \| v \|_{L^2(\tau)}^2 \\
    &\lesssim \rho^3 \sigma^{-1} d^3 h^{-1} \| v \|_{L^2(\Omega)}^2,
\end{split}
\end{equation*}
where the inequalities follow from \cref{Cor:trace}, \cref{Lem:inverse}, and quasi-uniformity, respectively.

While the above argument is sufficient for low dimensions, as in most standard applications, it becomes inadequate in high dimensions due to the high-order dependence on \(d\) introduced by the trace and inverse inequalities.
Therefore, in what follows, we provide a sharper estimate based on a more direct argument.

\begin{lemma}[discrete trace inequality]
\label{Lem:trace_discrete}
Let $\mathcal{T}_h$ be a $(\rho, \sigma)$-quasi-uniform triangulation of $\Omega \subset \mathbb{R}^d$.
Then we have
\begin{equation*}
    \| v \|_{L^2 (\partial \Omega )} \lesssim \rho^{\frac{1}{2}} \sigma^{-\frac{1}{2}} d h^{-\frac{1}{2}} \| v \|_{L^2 (\Omega)}, \quad v \in V_h.
\end{equation*}
\end{lemma}
\begin{proof}
Take any $\tau \in \mathcal{T}_h$.
Note that $v |_{\tau} \in \mathbb{P}_1 (\tau)$ for $v \in V_h$.
The $L^2 (\tau)$-mass matrix $\mathsf{M}_{\tau}$ for the nodal basis $\{ \phi_{\tau, a} \}_{a \in \mathcal{V} (\tau)}$ of the space $\mathbb{P}_1 (\tau)$ is given by
\begin{equation}
\label{Lem1:trace_discrete}
    \mathsf{M}_{\tau} = \frac{| \tau |}{(d+1)(d+2)} ( \mathsf{I}_{d+1} + \mathsf{1}_{d+1} \mathsf{1}_{d+1}^{\mathsf T}).
\end{equation}
On the other hand, since
\begin{equation*}
    \int_{F_a} \phi_{\tau, b} \phi_{\tau, c} \,dx
    = \begin{cases}
    \frac{|F_a|}{d(d+1)} (1 + \delta_{bc}), & \text{ if } b,c \neq a, \\
    0, & \text{ if } b = a \text{ or } c =  a,
    \end{cases}
    \quad a, b, c \in \mathcal{V} (\tau),
\end{equation*}
where $F_a$ denotes the $(d-1)$-face of $\tau$ opposite to $a$, the $L^2 (\partial \tau)$-mass matrix $\mathsf{M}_{\partial \tau}$ satisfies
\begin{multline}
\label{Lem2:trace_discrete}
    (\mathsf{M}_{\partial \tau})_{bc} = \sum_{a \in \mathcal{V} (\tau) \setminus \{ b,c \}} \int_{F_a} \phi_{\tau,b} \phi_{\tau,c} \,dx \\
    \eqsim d^{-2} (1 + \delta_{bc}) \sum_{a \in \mathcal{V} (\tau) \setminus \{ b,c \}} | F_a |
    \lesssim \rho h_{\tau}^{-1} | \tau | (1 + \delta_{bc}),
    \quad b,c \in \mathcal{V} (\tau),
\end{multline}
where the last inequality follows from $\# (\mathcal{V} (\tau) ) = d+1$, and from the following estimate, which is a direct consequence of \cref{Lem:shape-regular}:
\begin{equation*}
    |F_a| = d h_{F_a}^{-1} | \tau| \lesssim \rho d h_{\tau}^{-1} | \tau |.
\end{equation*}
Consequently, by~\eqref{Lem1:trace_discrete} and~\eqref{Lem2:trace_discrete}, we obtain
\begin{equation}
\label{Lem3:trace_discrete}
    \sup_{v \in \mathbb{P}_1 (\tau)} \frac{\| v \|_{L^2 (\partial \tau)}^2}{\| v \|_{L^2 (\tau)}^2}
    = \sup_{\mathsf{v} \in \mathbb{R}^{d+1} \setminus \{ \mathsf{0}_{d+1} \}} \frac{\mathsf{v}^{\mathsf T} \mathsf{M}_{\partial \tau} \mathsf{v}}{\mathsf{v}^{\mathsf T} \mathsf{M}_{\tau} \mathsf{v}} \lesssim \rho d^2 h_{\tau}^{-1}.
\end{equation}

Now, we take any $v \in V_h$.
It follows that
\begin{equation*}
    \| v \|_{L^2 (\partial \Omega)}^2
    \leq \sum_{\tau \in \mathcal{T}_h (\partial \Omega)} \| v \|_{L^2 (\partial \tau)}^2
    \stackrel{\eqref{Lem3:trace_discrete}}{\lesssim} \sum_{\tau \in \mathcal{T}_h (\partial \Omega)} \rho d^2 h_{\tau}^{-1} \| v \|_{L^2 (\tau)}^2
    \leq \rho \sigma^{-1} d^2 h^{-1} \| v \|_{L^2 (\Omega)}^2,
\end{equation*}
where the last inequality is due to the quasi-uniformity of $\mathcal{T}_h$.
\end{proof}

\subsection{Inverse inequalities}
Finally, we present inverse inequalities for finite element functions, together with their dependence on the spatial dimension \(d\).
Explicit constants in finite element inverse inequalities have been studied in the literature; see, e.g.,~\cite{CZ:2013}.
We begin with a sharp inverse inequality on the reference element in $\mathbb{R}^d$.

\begin{lemma}[reference inverse inequality]
\label{Lem:inverse_ref}
Let $\hat{\tau}_d$ be the $d$-dimensional reference element defined by
\begin{equation*}
\hat{\tau}_d = \left\{ x = (x_1, \dots, x_d) \in \mathbb{R}^d : x_i > 0, \sum_{i=1}^{d} x_i < 1 \right\} \subset \mathbb{R}^d.
\end{equation*}
Then we have
\begin{equation*}
\sup_{v \in \mathbb{P}_1 (\hat{\tau}_d) \setminus \{ 0 \} } \frac{| v |_{H^1 (\hat{\tau}_d)}}{\| v \|_{L^2 (\hat{\tau}_d)}}
= (d+1) (d+2)^{\frac{1}{2}}
\eqsim d^{\frac{3}{2}}.
\end{equation*}
\end{lemma}
\begin{proof}
Take any $v \in \mathbb{P}_1(\hat{\tau}_d)$ and write
\begin{equation*}
    v = \sum_{i=0}^d v_i \hat{\phi}_i,
\end{equation*}
where $\{ \hat{\phi}_i \}_{i=0}^d$ denotes the set of nodal basis functions on the reference element $\hat{\tau}_d$.
As in~\eqref{Lem1:trace_discrete}, by direct computation, we obtain
\begin{subequations}
\label{Lem1:inverse_ref}
\begin{equation}
\| v \|_{L^2(\hat{\tau}_d)}^2
= \frac{|\hat{\tau}_d|}{(d+1)(d+2)} \, \mathsf{v}^{\mathsf T} \mathsf{M}_d \mathsf{v},
\quad
| v |_{H^1(\hat{\tau}_d)}^2 
= |\hat{\tau}_d| \, \mathsf{v}^{\mathsf T} \mathsf{K}_d \mathsf{v},
\end{equation}
where
\begin{equation}
\mathsf{M}_d = \mathsf{I}_{d+1} + \mathsf{1}_{d+1} \mathsf{1}_{d+1}^{\mathsf T}, 
\quad
\mathsf{K}_d =
\begin{bmatrix}
d & -\mathsf{1}_d^{\mathsf T} \\
-\mathsf{1}_d & \mathsf{I}_d
\end{bmatrix}.
\end{equation}
\end{subequations}
Since one can verify by straightforward calculation that $\lambda_{\max}(\mathsf{M}_d^{-1} \mathsf{K}_d) = d + 1$, it follows from~\eqref{Lem1:inverse_ref} that the desired result holds.
\end{proof}

Using the reference inverse inequality presented in \cref{Lem:inverse_ref}, we can derive a general inverse inequality for elements in a shape-regular triangulation, as stated in \cref{Lem:inverse}.

\begin{lemma}[local inverse inequality]
\label{Lem:inverse}
Let $\mathcal{T}_h$ be a $\rho$-shape-regular triangulation of $\Omega \subset \mathbb{R}^d$.
Then for any $\tau \in \mathcal{T}_h$, we have
\begin{equation*}
|v|_{H^1 (\tau)} \lesssim \rho d^{\frac{3}{2}} h_{\tau}^{-1} \|v\|_{L^2 (\tau)},
\quad v \in \mathbb{P}_1 (\tau).
\end{equation*}
\end{lemma}
\begin{proof}
Let $F_\tau \colon \hat{\tau}_d \to \tau$ be the affine mapping given by 
$F_\tau(\hat{x}) = \mathsf{B}_{\tau} \hat{x} + \mathsf{b}_\tau$, 
where $\mathsf{B}_{\tau}$ is a $d \times d$ matrix and $\mathsf{b}_\tau$ is a vector.
The inball of $\tau$ with radius $r_{\tau}$ pulls back under $\mathsf{B}_{\tau}^{-1}$ to an ellipsoid contained in $\hat{\tau}_d$.
The largest axis of this ellipsoid is $2 r_{\tau} \| \mathsf{B}_{\tau}^{-1} \|$, which cannot exceed the diameter 
$h_{\hat{\tau}_d} = \sqrt{2}$ of $\hat{\tau}_d$.
Hence, we obtain
\begin{equation}
\label{Lem1:inverse_local}
\| \mathsf{B}_{\tau}^{-1} \| \lesssim r_{\tau}^{-1} \le \rho h_{\tau}^{-1},
\end{equation}
where the last inequality follows from the shape-regularity.

Let $\hat{v} \in \mathbb{P}_1(\hat{\tau}_d)$ be the pullback of $v$, i.e., $\hat{v}(\hat{x}) = v(F_\tau(\hat{x}))$ for $\hat{x} \in \hat{\tau}_d$.
By change of variables, we have
\begin{multline*}
| v |_{H^1 (\tau)}^2
= | \det \mathsf{B}_{\tau} | \int_{\hat{\tau}_d} | \mathsf{B}_{\tau}^{-\mathsf{T}} \nabla \hat{v} |^2 \, d \hat{x} \\
\stackrel{\eqref{Lem1:inverse_local}}{\lesssim} 
\rho^2 h_{\tau}^{-2} | \det \mathsf{B}_{\tau} | \, | \hat{v} |_{H^1 (\hat{\tau}_d)}^2
\lesssim \rho^2 d^3 h_{\tau}^{-2} | \det \mathsf{B}_{\tau} | \, \| \hat{v} \|_{L^2 (\hat{\tau}_d)}^2
= \rho^2 d^3 h_{\tau}^{-2} \| v \|_{L^2 (\tau)}^2,
\end{multline*}
where the second inequality follows from \cref{Lem:inverse_ref}.
This completes the proof.
\end{proof}

The following global result is an immediate consequence of the local estimate.

\begin{corollary}[global inverse inequality]
\label{Cor:inverse}
Let $\mathcal{T}_h$ be a $(\rho, \sigma)$-quasi-uniform triangulation of $\Omega \subset \mathbb{R}^d$.
Then we have
\begin{equation*}
| v |_{H^1 (\Omega)} \lesssim \rho \sigma^{-1} d^{\frac{3}{2}} h^{-1} \| v \|_{L^2 (\Omega)},
\quad v \in V_h.
\end{equation*}
\end{corollary}

\section{Averaged Scott--Zhang interpolation}
\label{Sec:SZ}
In this section, we consider a variant of the Scott--Zhang interpolation~\cite{SZ:1990}, called the \textit{averaged} Scott--Zhang interpolation, whose dependence on the spatial dimension \(d\) is only polynomial.
This modification is motivated by the fact that the standard Scott--Zhang interpolation can have unfavorable dimension dependence in high-dimensional settings; see \cref{App:SZ} for a concrete lower-bound example.
Using this averaged interpolation, we also derive estimates for the \(L^2\)- and \(H^1\)-orthogonal projections onto the finite element space \(V_h\) with polynomial dependence on \(d\).

To begin, we state our assumptions on the domain $\Omega$ and its triangulation $\mathcal{T}_h$.

\begin{assumption}
\label{Ass:triangulation}
The domain $\Omega \subset \mathbb{R}^d$ is bounded, convex, and polyhedral.
The triangulation $\mathcal{T}_h$ of $\Omega$ satisfies the following:
\begin{itemize}
\item[(a)] $\mathcal{T}_h$ is $(\rho,\sigma)$-quasi-uniform, and every vertex patch $\omega_a$, $a \in \mathcal{V}_h$, is convex.
\item[(b)] For each \(a\in\mathcal V_h(\partial\Omega)\), let
\begin{equation*}
    \mathcal K_a
    =
    \{K\in\mathcal F_h(\partial\Omega):a\in\mathcal V(K)\}.
\end{equation*}
For \(K\in\mathcal K_a\), let \(\tau_K\in\mathcal T_h\) be the unique element such that \(K\in\mathcal F(\tau_K)\).
We assume that
\begin{equation*}
    |\omega_a|
    \lesssim
    h\sum_{K\in\mathcal K_a} \mu_{a,K},
    \quad \mu_{a,K} = \frac{|K|}{\# (\mathcal K_a\cap\mathcal F(\tau_K))}.
\end{equation*}
\end{itemize}
\end{assumption}

Under \cref{Ass:triangulation}(a), the setting satisfies the geometric
requirements used to obtain favorable dependence on the dimension \(d\), such
as the dimension-independent elliptic regularity estimate in
\cref{Lem:regularity} and the Bramble--Hilbert estimate in \cref{Lem:BH}.
On the other hand, \cref{Ass:triangulation}(b) requires that the boundary faces
containing a boundary vertex \(a\) have enough
\((d-1)\)-dimensional measure so that, after multiplication by \(h\), they
control the measure of the vertex patch \(\omega_a\). The factor $\# (\mathcal K_a\cap\mathcal F(\tau_K))^{-1}$
prevents overcounting when an element has more than one boundary face containing \(a\). 

The structured Freudenthal hierarchy on $\Omega=(0,1)^d$ described in \cref{Sec:Triangulations} provides an example satisfying \cref{Ass:triangulation}.
To verify part~(b), fix a boundary vertex $a$ and a $(d-1)$-dimensional face $F$ of $\overline{\Omega}$ containing $a$.
Orthogonal projection onto $F$ maps each incident Freudenthal simplex to a boundary simplex containing $a$.
Thus, the projection of $\omega_a$ is covered by the boundary faces $K\in\mathcal K_a$ lying in $F$.
The patch lies in a slab of thickness at most $h$ adjacent to $F$.
Moreover, each Freudenthal simplex has only two coordinate-aligned facets and hence at most two boundary faces.
Consequently, we have
\begin{equation*}
    |\omega_a|
    \le h\sum_{K\in\mathcal K_a}|K|
    \le 2h\sum_{K\in\mathcal K_a}\mu_{a,K}.
\end{equation*}
The constant $2$ is independent of both the dimension and the refinement level.

\subsection{Definitions}
The Scott--Zhang interpolation was first introduced in~\cite{SZ:1990} as a modified Lagrange-type interpolation for approximating nonsmooth functions in Sobolev spaces by continuous piecewise polynomials, while also preserving homogeneous Dirichlet boundary conditions.
For completeness, we present a version of the Scott--Zhang interpolation in \cref{Def:SZ}.
More general discussions can be found in~\cite{BS:2008,SZ:1990}.

\begin{definition}[Scott--Zhang interpolation]
\label{Def:SZ}
For each $a \in \mathcal{V}_h$, we define $K_a$ as follows:
\begin{itemize}
    \item If \(a\) is an interior node, \(K_a\in\mathcal T_h\) is chosen so that
\(a\in\mathcal V(K_a)\).
    \item If \(a\in\mathcal V_h(\partial\Omega)\), \(K_a\) is a \((d-1)\)-face such that \(a\in\mathcal V(K_a)\) and \(K_a\subset\partial\Omega\).
\end{itemize}
The Scott--Zhang interpolation $I_h^{\mathrm{SZ}} \colon H^1 (\Omega) \to V_h$ is defined by
\begin{equation*}
    (I_h^{\mathrm{SZ}} v) (a) = \int_{K_a} \psi_a v \,dx,
    \quad v \in H^1 (\Omega),\ a \in \mathcal{V}_h,
\end{equation*}
where $\psi_a \in \mathbb{P}_1 (K_a)$ is the function dual to the $K_a$-restricted nodal basis $\{ \phi_b |_{K_a} \}_{b \in \mathcal{V} (K_a)}$, i.e.,
\begin{equation*}
    \int_{K_a} \psi_a \phi_b \,dx = \delta_{ab},
    \quad b \in \mathcal{V} (K_a).
\end{equation*}
\end{definition}

While the Scott--Zhang interpolation has been successfully applied in the analysis of finite element methods and multilevel iterative methods (see~\cite{BS:2008,FMP:2021}), its standard estimates involve the number of elements sharing a given vertex. 
Since this number can grow rapidly with the dimension~$d$, the corresponding bounds may inherit an undesirable dependence on~$d$ in high-dimensional settings~(cf.~\cref{App:SZ}).

To address this issue, we use the \emph{averaged} Scott--Zhang interpolation, whose definition is given in \cref{Def:SZ_averaged}.

\begin{definition}[averaged Scott--Zhang interpolation]
\label{Def:SZ_averaged}
The averaged Scott--Zhang interpolation \(I_h \colon H^1(\Omega) \to V_h\) is defined by
\begin{equation*}
    (I_h v)(a)
    =
    \left( \sum_{K \in \mathcal{K}_a} \mu_{a,K} \right)^{-1}
    \sum_{K \in \mathcal{K}_a} \mu_{a,K}
    \int_K \psi_{a,K} v \,dx,
    \quad v \in H^1(\Omega),\ a \in \mathcal{V}_h,
\end{equation*}
where \(\psi_{a,K} \in \mathbb{P}_1(K)\) is the function dual to the
\(K\)-restricted nodal basis \(\{ \phi_b|_K \}_{b \in \mathcal{V}(K)}\).
The collection \(\mathcal K_a\) and the weight \(\mu_{a,K}\) are defined as follows:
\begin{itemize}
    \item If \(a \in \mathcal V_h\setminus\mathcal V_h(\partial\Omega)\), then
    \[
        \mathcal K_a=\{\tau\in\mathcal T_h: a\in\mathcal V(\tau)\},
        \quad
        \mu_{a,K}=|K|,
        \quad K\in\mathcal K_a.
    \]
    \item If \(a\in\mathcal V_h(\partial\Omega)\), then $\mathcal{K}_a$ and $\mu_{a,K}$ are given as in \cref{Ass:triangulation}.
\end{itemize}
\end{definition}

We next discuss connections with existing finite element quasi-interpolation operators.
The averaged Scott--Zhang interpolation resembles the Cl\'{e}ment interpolation~\cite{Clement:1975} in the sense that the value at a vertex is defined using information from the entire vertex patch.
It is also related to projective quasi-interpolation operators based on local projection and averaging~\cite{EG:2017,KPY:2018}.
Our emphasis is on explicitly tracking the dependence of the constants on \(d\), \(\rho\), and \(\sigma\), rather than on the novelty of the basic quasi-interpolation construction.

\subsection{Error estimates}
We now derive the $L^2$- and $H^1$-error estimates for the averaged Scott--Zhang interpolation.
As a first step, we present some useful estimates for the dual functions \(\psi_{a,K}\) appearing in \cref{Def:SZ_averaged}.

\begin{lemma}
\label{Lem:psi}
For each $a \in \mathcal{V}_h$, let $\mathcal{K}_a$ be defined as in \cref{Def:SZ_averaged}.
For each $K \in \mathcal{K}_a$, the dual function $\psi_{a, K} \in \mathbb{P}_1 (K)$ introduced in \cref{Def:SZ_averaged} satisfies
\begin{equation*}
    \| \psi_{a, K} \|_{L^1 (K)} \eqsim d, \quad
    \| \psi_{a, K} \|_{L^2 (K)}
    \eqsim d | K |^{-\frac{1}{2}}.
\end{equation*}
\end{lemma}
\begin{proof}
Note that $K$ is a $d$-simplex when $a \in \mathcal{V}_h \setminus \mathcal V_h(\partial\Omega)$, and a $(d-1)$-simplex when $a \in \mathcal V_h(\partial\Omega)$.
If $K$ is a $d$-simplex, then, as in~\eqref{Lem1:trace_discrete}, the $L^2 (K)$-mass matrix $\mathsf{M}_K$ associated with the basis
$\{ \phi_b|_K \}_{b \in \mathcal{V} ( K )}$
is given by
\begin{equation*}
    \mathsf{M}_K = \frac{| K |}{(d+1)(d+2)} (\mathsf{I}_{d+1} + \mathsf{1}_{d+1} \mathsf{1}_{d+1}^{\mathsf T}).
\end{equation*}
Hence, we have
\begin{equation*}
    \mathsf{M}_K^{-1} 
    = \frac{(d+1)(d+2)}{| K |} \left( \mathsf{I}_{d+1} - \frac{1}{d+2} \mathsf{1}_{d+1} \mathsf{1}_{d+1}^{\mathsf T} \right).
\end{equation*}
Since $\psi_{a, K}$ is the $L^2 (K)$-dual of $\phi_a|_K$, we obtain
\begin{equation*}
\psi_{a, K} = \sum_{b \in \mathcal{V} (K)} (\mathsf{M}_K^{-1})_{ba} \phi_b|_K
= \frac{d+1}{|K|} ( (d+2) \phi_a |_K - 1).
\end{equation*}
Direct calculation yields
\begin{align*}
    \| \psi_{a, K} \|_{L^1 (K)} &= \frac{d+1}{|K|} \int_{K} \big| (d+2) \phi_a |_K - 1 \big| \,dx = \frac{2 (d+1)^{d+1}}{(d+2)^d} - 1, \\
    \| \psi_{a, K} \|_{L^2 (K)}^2 &= (\mathsf{M}_K^{-1})_{aa} = \frac{(d+1)^2}{| K |},
\end{align*}
which is our desired result.

If $K$ is a $(d-1)$-simplex, an analogous argument using the $(d-1)$-dimensional mass matrix yields the desired result.
\end{proof}

Using the identities
\begin{equation}
\label{phi_psi}
    \sum_{a \in \mathcal{V}(\tau)} \phi_a(x) = 1, \quad
    \int_K \psi_{a,K}(y)\, dy = 1,
\end{equation}
we readily obtain \cref{Lem:SZ_averaged_error}, which will be useful in deriving $L^2$-error estimates for the averaged Scott--Zhang interpolation.

\begin{lemma}
\label{Lem:SZ_averaged_error}
For each \(\tau\in\mathcal T_h\), the averaged Scott--Zhang interpolation
\(I_h\colon H^1(\Omega)\to V_h\) defined in \cref{Def:SZ_averaged} satisfies
\[
    (v-I_hv)(x)
    =
    \sum_{a\in\mathcal V(\tau)}(\Psi_a v)(x)\phi_a(x)
\]
for a.e.\ \(x\in\tau\), where $( \Psi_a v ) (x)$ is given by
\begin{equation*}
    (\Psi_a v)(x)
    =
    \left( \sum_{K \in \mathcal K_a} \mu_{a,K} \right)^{-1}
    \sum_{K \in \mathcal K_a}
    \mu_{a,K}
    \int_K \psi_{a,K}(y)(v(x)-v(y))\,dy.
\end{equation*}
\end{lemma}

Thanks to \cref{Lem:SZ_averaged_error}, it suffices to estimate \(\Psi_a v\) for each \(a \in \mathcal{V}_h\) in order to obtain an estimate for the error \(v - I_h v\) of the averaged Scott--Zhang interpolation; a similar technique was used, for example, in~\cite[Lemma~5.1]{Park:2024}.
In \cref{Lem:Psi_L2}, we present an \(L^2\)-estimate for \(\Psi_a v\).

\begin{lemma}
\label{Lem:Psi_L2}
Suppose that \cref{Ass:triangulation} holds.
Then, for each $a \in \mathcal{V}_h$, the operator $\Psi_a$ defined in \cref{Lem:SZ_averaged_error} satisfies
\begin{equation*}
    \| \Psi_a v \|_{L^2 (\omega_a)}
    \lesssim
    \rho^{\frac{1}{2}} \sigma^{-\frac{1}{2}} d^{\frac{3}{2}} h | v |_{H^1 (\omega_a)},
    \quad v \in H^1 (\Omega).
\end{equation*}
\end{lemma}

\begin{proof}
It suffices to prove the estimate for \(v\in C^1(\overline{\Omega})\), since the general case \(v\in H^1(\Omega)\) follows by density.
We define~(see \cref{Def:SZ_averaged} for the definitions of $\mathcal{K}_a$ and $\mu_{a,K}$)
\begin{equation*}
    \alpha_K
    =
    \left(\sum_{K' \in \mathcal{K}_a } \mu_{a,K'}\right)^{-1}\mu_{a,K},
    \quad K \in \mathcal{K}_a,
\end{equation*}
so that we have
\begin{equation*}
    0 \leq \alpha_K \leq 1, \quad \sum_{K \in \mathcal{K}_a} \alpha_K = 1.
\end{equation*}
From the definition of $\Psi_a v$, the Jensen inequality
\begin{equation*}
\varphi \left( \sum_{K \in \mathcal{K}_a} \alpha_K v_K \right)
\leq \sum_{K \in \mathcal{K}_a} \alpha_K \varphi (v_K)
\end{equation*}
with
\begin{equation*}
    \varphi = \| \cdot \|_{L^2 (\omega_a)}^2, \quad
    v_K = \int_K \psi_{a,K} (y) (v (\cdot) - v(y)) \,dy
\end{equation*}
gives
\begin{equation*}
\begin{split}
    \| \Psi_a v \|_{L^2 (\omega_a)}^2
   &= \left\| \sum_{K \in \mathcal{K}_a} \alpha_K \int_K \psi_{a, K} (y) ( v(\cdot) - v(y) ) \,dy \right\|_{L^2 (\omega_a)}^2 \\
   &\leq \sum_{K \in \mathcal{K}_a} \alpha_K
   \left\| \int_K \psi_{a, K} (y) ( v( \cdot ) - v(y) ) \,dy \right\|_{L^2 (\omega_a)}^2.
\end{split}
\end{equation*}
It follows that
\begin{equation}
\label{Lem1:Psi_L2}
\begin{split}
   \| \Psi_a v \|_{L^2 (\omega_a)}^2
   &\leq \sum_{K \in \mathcal{K}_a} \alpha_K \int_{\omega_a}
   \left( \int_K \psi_{a, K} (y)^2 \,dy \right)
   \left( \int_K ( v(x) - v(y) )^2 \,dy \right) \,dx \\
   &\lesssim d^2 \left( \sum_{K \in \mathcal{K}_a} \mu_{a,K} \right)^{-1}
   \sum_{K \in \mathcal{K}_a}
   \frac{\mu_{a,K}}{|K|}
   \int_{\omega_a} \int_K
   ( v(x) - v(y) )^2 \,dy \,dx,
\end{split}
\end{equation}
where we used the Cauchy--Schwarz inequality and \cref{Lem:psi}.

Now, we consider two cases $a \in \mathcal{V}_h \setminus \mathcal V_h(\partial \Omega)$ and $a \in \mathcal{V}_h(\partial \Omega)$ separately.
First, we consider the case $a \in \mathcal{V}_h \setminus \mathcal V_h(\partial \Omega)$.
In this case, $\mathcal{K}_a$ consists of the elements contained in $\omega_a$ and $\mu_{a,K}=|K|$.
Hence,
\[
    \sum_{K\in\mathcal K_a}\mu_{a,K}=|\omega_a|.
\]
Thus,~\eqref{Lem1:Psi_L2} implies
\begin{equation}
\label{Lem2:Psi_L2}
    \| \Psi_a v \|_{L^2 (\omega_a)}^2
    \lesssim d^2 | \omega_a |^{-1}
    \int_{\omega_a} \int_{\omega_a} ( v(x) - v(y) )^2 \,dy \,dx
    \lesssim d^2 \| v \|_{L^2 (\omega_a)}^2,
\end{equation}
where the last inequality uses
\begin{multline*}
    \int_{\omega_a} \int_{\omega_a} ( v(x) - v(y) )^2 \,dy \,dx \\
    \leq 2 \int_{\omega_a} \int_{\omega_a} v(x)^2 \,dy \,dx
    + 2 \int_{\omega_a} \int_{\omega_a} v(y)^2 \,dy \,dx
    = 4 | \omega_a | \int_{\omega_a} v(x)^2 \,dx.
\end{multline*}
Note that $\Psi_a$ is invariant under addition of a constant. Namely,
\begin{equation}
\label{Lem4:Psi_L2}
\Psi_a v = \Psi_a (v + c),
\quad c \in \mathbb{R}.
\end{equation}
Therefore, from~\eqref{Lem2:Psi_L2} and~\eqref{Lem4:Psi_L2}, we deduce
\begin{equation*}
    \| \Psi_a v \|_{L^2 (\omega_a)}^2
    \lesssim d^2 \inf_{c \in \mathbb{R}} \| v + c \|_{L^2 (\omega_a)}^2
    \lesssim d^2 h^{2} | v |_{H^1 ( \omega_a )}^2,
\end{equation*}
where the last inequality is due to \cref{Lem:Poincare} with \cref{Ass:triangulation}(a).

Next, we consider the case $a \in \mathcal{V}_h(\partial \Omega)$.
In this case,
\[
    \mathcal K_a
    =
    \{K\in\mathcal F_h(\partial\Omega):a\in\mathcal V(K)\},
    \quad
    \mu_{a,K}
    =
    \frac{|K|}{\#(\mathcal K_a\cap\mathcal F(\tau_K))}.
\]
Then~\eqref{Lem1:Psi_L2} becomes
\begin{equation}
\label{Lem5:Psi_L2}
\resizebox{\textwidth}{!}{$\displaystyle
   \| \Psi_a v \|_{L^2 (\omega_a)}^2
   \lesssim d^2
   \left(
   \sum_{K\in\mathcal K_a}
   \frac{|K|}{\#(\mathcal K_a\cap\mathcal F(\tau_K))}
   \right)^{-1}
   \sum_{K\in\mathcal K_a}
   \frac{1}{\#(\mathcal K_a\cap\mathcal F(\tau_K))}
   \int_{\omega_a}\int_K
   ( v(x)-v(y) )^2\,dy\,dx .
   $}
\end{equation}
For fixed \(x\in\omega_a\), applying \cref{Lem:trace} to the function
\(y\mapsto v(x)-v(y)\) on \(\tau_K\), and using quasi-uniformity, gives
\begin{equation}
\label{Lem6:Psi_L2}
\begin{split}
    \int_K ( v(x)-v(y) )^2\,dy
    &\lesssim
    \rho d h_{\tau_K}^{-1}
    \int_{\tau_K} ( v(x)-v(y) )^2\,dy
    +
    \rho d^{-1} h_{\tau_K} |v|_{H^1(\tau_K)}^2 \\
    &\lesssim
    \rho\sigma^{-1} d h^{-1}
    \int_{\tau_K} ( v(x)-v(y) )^2\,dy
    +
    \rho d^{-1} h |v|_{H^1(\tau_K)}^2 .
\end{split}
\end{equation}
Combining~\eqref{Lem5:Psi_L2},~\eqref{Lem6:Psi_L2}, and \cref{Ass:triangulation}(b), we obtain
\begin{equation*}
\begin{split}
    \| \Psi_a v \|_{L^2(\omega_a)}^2
    &\lesssim
    \rho\sigma^{-1} d^3 |\omega_a|^{-1}
    \int_{\omega_a}
    \sum_{K\in\mathcal K_a}
    \frac{1}{\#(\mathcal K_a\cap\mathcal F(\tau_K))}
    \int_{\tau_K}
    (v(x)-v(y))^2\,dy\,dx \\
    &\quad+
    \rho d h^2
    \sum_{K\in\mathcal K_a}
    \frac{1}{\#(\mathcal K_a\cap\mathcal F(\tau_K))}
    |v|_{H^1(\tau_K)}^2 \\
    &\leq
    \rho\sigma^{-1} d^3 |\omega_a|^{-1}
    \int_{\omega_a}\int_{\omega_a}
    (v(x)-v(y))^2\,dy\,dx +
    \rho d h^2 |v|_{H^1(\omega_a)}^2.
\end{split}
\end{equation*}
By the same argument used in the interior case and by
\cref{Lem:Poincare}, we get
\[
    |\omega_a|^{-1}
    \int_{\omega_a}\int_{\omega_a}
    (v(x)-v(y))^2\,dy\,dx
    \lesssim
    h^2 |v|_{H^1(\omega_a)}^2 .
\]
Hence,
\begin{equation*}
    \| \Psi_a v \|_{L^2(\omega_a)}^2
    \lesssim
    \rho\sigma^{-1}d^3h^2 |v|_{H^1(\omega_a)}^2,
\end{equation*}
which is the desired result.
\end{proof}

Using \cref{Lem:Psi_L2}, we now present an \(L^2\)-error estimate for the averaged Scott--Zhang interpolation in \cref{Thm:SZ_averaged_L2}.

\begin{theorem}[$L^2$-error estimate]
\label{Thm:SZ_averaged_L2}
Suppose that \cref{Ass:triangulation} holds.
Then the averaged Scott--Zhang interpolation $I_h \colon H^1 (\Omega) \to V_h$ defined in~\cref{Def:SZ_averaged} satisfies
\begin{equation*}
    \| v - I_h v \|_{L^2 (\Omega)}
    \lesssim
    \rho^{\frac{1}{2}} \sigma^{-\frac{1}{2}} d^{2} h |v|_{H^1 (\Omega)},
    \quad v \in H^1 (\Omega).
\end{equation*}
\end{theorem}
\begin{proof}
Take any element $\tau \in \mathcal{T}_h$ and $x \in \tau$.
By \cref{Lem:SZ_averaged_error} and the Cauchy--Schwarz inequality, we get
\begin{multline*}
(v - I_h v) (x)^2
= \left( \sum_{a \in \mathcal{V} (\tau)} ( \Psi_a v) (x) \phi_a (x) \right)^2 \\
\leq \left( \sum_{a \in \mathcal{V} (\tau)} (\Psi_a v)(x)^2 \right)
\left( \sum_{a \in \mathcal{V} (\tau)} \phi_a (x)^2 \right)
\stackrel{\eqref{phi_psi}}{\leq} \sum_{a \in \mathcal{V} (\tau)} (\Psi_a v)(x)^2.
\end{multline*}
Integrating over $x \in \tau$ followed by summing over $\tau \in \mathcal{T}_h$ yields
\begin{multline*}
    \| v - I_h v \|_{L^2 (\Omega)}^2
    = \sum_{\tau \in \mathcal{T}_h} \| v - I_h v \|_{L^2 (\tau)}^2
    \leq \sum_{\tau \in \mathcal{T}_h} \sum_{a \in \mathcal{V} (\tau)}
    \| \Psi_a v \|_{L^2 (\tau)}^2 \\
    = \sum_{a \in \mathcal{V}_h} \| \Psi_a v \|_{L^2 (\omega_a)}^2
    \lesssim
    \rho \sigma^{-1} d^{3} h^{2}
    \sum_{a \in \mathcal{V}_h} | v |_{H^1 (\omega_a)}^2
    \lesssim
    \rho \sigma^{-1} d^{4} h^{2} | v |_{H^1 (\Omega)}^2,
\end{multline*}
where the penultimate inequality is due to \cref{Lem:Psi_L2}, and the last inequality is because each element has $d+1$ vertices.
\end{proof}

Next, we consider the $H^1$-error estimate for the averaged Scott--Zhang interpolation, as stated in \cref{Thm:SZ_averaged_H1}.

\begin{theorem}[$H^1$-error estimate]
\label{Thm:SZ_averaged_H1}
Suppose that \cref{Ass:triangulation} holds.
Then, the averaged Scott--Zhang interpolation $I_h \colon H^1 (\Omega) \to V_h$ defined in~\cref{Def:SZ_averaged} satisfies
\begin{equation*}
    | v - I_h v |_{H^1 (\Omega)}
    \lesssim
    \rho^{\frac{3}{2}} \sigma^{-\frac{3}{2}} d^2 h |v|_{H^2 (\Omega)},
    \quad v \in H^2 (\Omega).
\end{equation*}
\end{theorem}
\begin{proof}
It suffices to prove the estimate for \(v\in C^\infty(\overline{\Omega})\), since the general case \(v\in H^2(\Omega)\) follows by density.
Take any element $\tau \in \mathcal{T}_h$ and $x \in \tau$.
For any $a \in \mathcal{V} (\tau)$, from the definition of $\Psi_a v$, we have
\begin{equation}
\label{Thm1:SZ_averaged_H1}
\nabla (\Psi_a v) = \nabla v \quad \text{ in } \tau.
\end{equation}
It follows that
\begin{equation}
\label{Thm2:SZ_averaged_H1}
\begin{split}
    | \nabla (v - I_hv ) (x) |^2
    &\leq
    \left|
    \sum_{a \in \mathcal{V} (\tau)}
    \nabla ( (\Psi_a v) \phi_a) (x)
    \right|^2 \\
    &{\lesssim}
    | \nabla v(x) |^2
    +
    \left|
    \sum_{a \in \mathcal{V}(\tau)}
    (\Psi_a v)(x)\nabla\phi_a (x)
    \right|^2 \\
    &\lesssim
    | \nabla v(x) |^2
    +
    \rho^2\sigma^{-2}h^{-2}
    \sum_{a \in \mathcal{V}(\tau)}
    (\Psi_a v)(x)^2 .
\end{split}
\end{equation}
Here, the first inequality is due to \cref{Lem:SZ_averaged_error}.
The second inequality follows from the Cauchy--Schwarz inequality,~\eqref{Thm1:SZ_averaged_H1}, and
$\sum_{a\in\mathcal{V}(\tau)}\phi_a=1$ on $\tau$. 
For the last inequality, set
\[
    q(y)=\sum_{a\in\mathcal{V}(\tau)}(\Psi_a v)(x)\phi_a(y),
    \quad y\in\tau .
\]
Then $q\in\mathbb{P}_1(\tau)$ and
\[
    \nabla q
    =
    \sum_{a\in\mathcal{V}(\tau)}
    (\Psi_a v)(x)\nabla\phi_a .
\]
Since an affine function attains its maximum and minimum at the vertices, and since $\tau$ contains a ball of radius $r_\tau$, we have
\[
    |\nabla q (x)|
    \lesssim
    r_\tau^{-1}
    \max_{a\in\mathcal{V}(\tau)} |(\Psi_a v)(x)|
    \leq
    r_\tau^{-1}
    \left(
    \sum_{a\in\mathcal{V}(\tau)}(\Psi_a v)(x)^2
    \right)^{\frac{1}{2}}.
\]
Using $r_\tau^{-1}\leq \rho h_\tau^{-1}\leq \rho\sigma^{-1}h^{-1}$ (see \cref{Def:shape-regular,Def:quasi-uniform}) gives the desired estimate.

Invoking \cref{Lem:BH}, for each \(a \in \mathcal{V}_h\), choose \(p_a \in \mathbb{P}_1(\Omega)\) such that
\begin{equation}
\label{p_a}
    | v - p_a |_{H^1(\omega_a)}
    \lesssim h |v|_{H^2(\omega_a)} .
\end{equation}
Since \(v-I_hv\) is invariant under addition of affine polynomials, applying
the pointwise estimate~\eqref{Thm2:SZ_averaged_H1} on each element \(\tau\),
with \(v\) replaced by
\(v - \frac{1}{d+1}\sum_{b\in\mathcal{V}(\tau)}p_b\), gives
\begin{multline}
\label{Thm4:SZ_averaged_H1}
    | \nabla (v - I_hv ) (x) |^2
    \lesssim
    \left|
    \nabla\left(
    v-\frac{1}{d+1}\sum_{b\in\mathcal{V}(\tau)}p_b
    \right)(x)
    \right|^2 \\
    + \rho^2\sigma^{-2}h^{-2}
    \sum_{a \in \mathcal{V}(\tau)}
    \left(
    \Psi_a\left(
    v-\frac{1}{d+1}\sum_{b\in\mathcal{V}(\tau)}p_b
    \right)(x)
    \right)^2 .
\end{multline}
Integrating~\eqref{Thm4:SZ_averaged_H1} over $x \in \tau$ gives
\begin{equation}
\label{Thm5:SZ_averaged_H1}
\begin{split}
| v - I_h v |_{H^1 (\tau)}^2
&\lesssim
\left| v - \frac{1}{d+1} \sum_{b \in \mathcal{V} (\tau)} p_b \right|_{H^1 (\tau)}^2 \\
&\quad
+ \rho^2 \sigma^{-2} h^{-2}
\sum_{a \in \mathcal{V} (\tau)}
\left\| \Psi_a \left( v - \frac{1}{d+1} \sum_{b \in \mathcal{V} (\tau)} p_b \right) \right\|_{L^2 (\tau)}^2 .
\end{split}
\end{equation}
We estimate the two terms in the right-hand side of~\eqref{Thm5:SZ_averaged_H1} separately.
By the Jensen inequality, we have
\begin{equation}
\label{Thm6:SZ_averaged_H1}
    \left|
    v-\frac{1}{d+1}\sum_{b\in\mathcal{V}(\tau)}p_b
    \right|_{H^1(\tau)}^2
    \leq
    \frac{1}{d+1}
    \sum_{a \in \mathcal{V}(\tau)}
    |v-p_a|_{H^1(\tau)}^2
    \leq
    \sum_{a \in \mathcal{V}(\tau)}
    |v-p_a|_{H^1(\tau)}^2 .
\end{equation}
Moreover, since
\begin{equation*}
\Psi_a \left(
v - \frac{1}{d+1} \sum_{b \in \mathcal{V}(\tau)} p_b
\right)
= \Psi_a(v-p_a) + q_a - q_a(a),
\end{equation*}
where
\begin{equation*}
q_a
= p_a - \frac{1}{d+1} \sum_{b \in \mathcal{V}(\tau)} p_b,
\end{equation*}
and
\begin{equation*}
\| q_a - q_a(a) \|_{L^2(\tau)}
\leq h_{\tau} |q_a|_{H^1(\tau)},
\qquad q_a \in \mathbb{P}_1(\tau),
\end{equation*}
we obtain
\begin{equation}
\label{Thm7:SZ_averaged_H1}
\begin{aligned}
\sum_{a \in \mathcal{V}(\tau)}
&\left\|
\Psi_a \left(
v-\frac{1}{d+1}\sum_{b\in\mathcal{V}(\tau)}p_b
\right)
\right\|_{L^2(\tau)}^2
\\
&\lesssim
\sum_{a \in \mathcal{V}(\tau)}
\| \Psi_a(v-p_a) \|_{L^2(\tau)}^2
+
\sum_{a \in \mathcal{V}(\tau)}
\| q_a - q_a(a) \|_{L^2(\tau)}^2
\\
&\leq
\sum_{a \in \mathcal{V}(\tau)}
\| \Psi_a(v-p_a) \|_{L^2(\tau)}^2
+
h^2
\sum_{a \in \mathcal{V}(\tau)}
|q_a|_{H^1(\tau)}^2
\\
&\le
\sum_{a \in \mathcal{V}(\tau)}
\| \Psi_a(v-p_a) \|_{L^2(\tau)}^2
+
h^2
\sum_{a \in \mathcal{V}(\tau)}
|v-p_a|_{H^1(\tau)}^2 .
\end{aligned}
\end{equation}
Here the last inequality follows from the fact that, for each fixed
$x \in \tau$, the average of the vectors
$\{\nabla p_a\}_{a \in \mathcal{V}(\tau)}$ minimizes the sum of squared
distances to these vectors.
Combining~\eqref{Thm5:SZ_averaged_H1},~\eqref{Thm6:SZ_averaged_H1}, and~\eqref{Thm7:SZ_averaged_H1}, we get
\begin{equation}
\label{Thm8:SZ_averaged_H1}
|v - I_h v|_{H^1 (\tau)}^2
\lesssim
\rho^2 \sigma^{-2}
\sum_{a \in \mathcal{V} (\tau)}
\left(
h^{-2} \| \Psi_a (v - p_a) \|_{L^2 (\tau)}^2
+ |v - p_a |_{H^1 (\tau)}^2
\right).
\end{equation}
Summing~\eqref{Thm8:SZ_averaged_H1} over $\tau \in \mathcal{T}_h$ and rearranging the sums yields
\begin{equation*}
\begin{split}
|v - I_h v |_{H^1 (\Omega)}^2
&\lesssim
\rho^2 \sigma^{-2}
\sum_{a \in \mathcal{V}_h}
\left(
h^{-2} \| \Psi_a (v - p_a) \|_{L^2 (\omega_a)}^2
+ |v - p_a |_{H^1 (\omega_a)}^2
\right) \\
&\stackrel{\text{(i)}}{\lesssim}
\rho^3 \sigma^{-3} d^3
\sum_{a \in \mathcal{V}_h}
|v - p_a |_{H^1 (\omega_a)}^2 \\
&\stackrel{\eqref{p_a}}{\lesssim}
\rho^3 \sigma^{-3} d^3 h^2
\sum_{a \in \mathcal{V}_h}
|v |_{H^2 (\omega_a)}^2 \\
&\stackrel{\text{(ii)}}{\lesssim}
\rho^3 \sigma^{-3} d^4 h^2
|v |_{H^2 (\Omega)}^2,
\end{split}
\end{equation*}
where the inequalities (i) and (ii) follow from \cref{Lem:Psi_L2} and the fact that each element has $d+1$ vertices, respectively.
This completes the proof.
\end{proof}

\subsection{Estimates for orthogonal projections}
We are now ready to study the \(L^2\)- and \(H^1\)-orthogonal projections, which play important roles in the analysis of multilevel methods, using the averaged Scott--Zhang interpolation.

We first consider the $L^2$-orthogonal projection, which has been studied in, e.g.,~\cite{BX:1991,Xu:1991}.
Let $ Q_h \colon H^1(\Omega) \to V_h$ be the $ L^2 $-orthogonal projection onto $V_h$:
\begin{equation}
\label{Q_h}
    \int_{\Omega} ( Q_h v ) v_h \,dx  = \int_{\Omega} v v_h \,dx, \quad v \in H^1 (\Omega),\ v_h \in V_h.
\end{equation}
The following estimate can be obtained straightforwardly from the $L^2$-approximation estimate of the averaged Scott--Zhang interpolation.

\begin{theorem}
\label{Thm:Q_h}
Suppose that \cref{Ass:triangulation} holds.
Then the $L^2$-orthogonal projection $Q_h \colon H^1 (\Omega) \to V_h$ defined in~\eqref{Q_h} satisfies
\begin{equation*}
    \| v - Q_h v \|_{L^2 (\Omega)}
    \lesssim \rho^{\frac{1}{2}} \sigma^{-\frac{1}{2}} d^2 h | v |_{H^1 (\Omega)},
    \quad v \in H^1 (\Omega).
\end{equation*}
\end{theorem}
\begin{proof}
Given $v \in H^1 (\Omega)$, by \cref{Thm:SZ_averaged_L2}, we have
\begin{equation*}
    \| v - Q_h v \|_{L^2 (\Omega)}
    = \inf_{v_h \in V_h} \| v - v_h \|_{L^2 (\Omega)}
    \leq \| v - I_h v \|_{L^2 (\Omega)}
    \lesssim \rho^{\frac{1}{2}} \sigma^{-\frac{1}{2}} d^2 h | v |_{H^1 (\Omega)},
\end{equation*}
where $I_h$ is the averaged Scott--Zhang interpolation defined in \cref{Def:SZ_averaged}.
\end{proof}

Since the averaged Scott--Zhang operator preserves homogeneous Dirichlet
boundary conditions, the same estimate in~\cref{Thm:Q_h} holds for the \(L^2\)-orthogonal projection \(Q_{h,0} \colon H_0^1(\Omega)\to V_{h,0}\).

Similarly, let $ P_h \colon H_0^1(\Omega) \to  V_{h,0} $ be the $ H^1 $-orthogonal projection onto $ V_{h,0} $:
\begin{equation}
\label{P_h}
    \int_{\Omega} \nabla(P_{h} v) \cdot \nabla v_h \,dx  = \int_{\Omega} \nabla v \cdot \nabla v_h \,dx, \quad v \in H_0^1 (\Omega),\ v_h \in V_{h,0}.
\end{equation}
Note that $P_h$ is well-defined thanks to the essential boundary conditions imposed on $H_0^1 (\Omega)$ and $V_{h,0}$.
Using a standard duality argument~(see, e.g.,~\cite{BS:2008}) and the $H^1$-approximation estimate of the averaged Scott--Zhang interpolation, we obtain the following result.

\begin{theorem}
\label{Thm:P_h}
Suppose that \cref{Ass:triangulation} holds.
Then the $H^1$-orthogonal projection $P_h \colon H_0^1 (\Omega) \to V_{h,0}$ defined in~\eqref{P_h} satisfies
\begin{equation*}
\| v - P_h v\|_{L^2 (\Omega)}
\lesssim
\rho^{\frac{3}{2}} \sigma^{-\frac{3}{2}} d^2 h | v |_{H^1 (\Omega)},
\quad v \in H_0^1 (\Omega).
\end{equation*}
\end{theorem}
\begin{proof}
Take any $v\in H_0^1(\Omega)$.
Consider the auxiliary problem: find $w\in H_0^1(\Omega)$ such that
\begin{equation}
\label{Thm1:P_h}
\int_{\Omega} \nabla w \cdot \nabla \phi \,dx = \int_{\Omega} (v-P_h v) \phi \,dx ,
\quad \phi \in H_0^1 (\Omega).
\end{equation}
By \cref{Lem:regularity}, we have
\begin{equation}
\label{Thm2:P_h}
    |w|_{H^2 (\Omega)} \leq \|v-P_h v\|_{L^2 (\Omega)}.
\end{equation}
Therefore, for any $v_h \in V_{h,0}$, we get
\begin{equation*}
    \|v-P_h v\|^2_{L^2 (\Omega)}
    \stackrel{\eqref{Thm1:P_h}}{=} \int_{\Omega} \nabla w \cdot \nabla (v-P_h v) \,dx
    \stackrel{\eqref{P_h}}{=} \int_{\Omega} \nabla (w-v_h) \cdot \nabla (v-P_h v) \,dx.
\end{equation*}
Choosing \(v_h = I_h w\) gives
\begin{equation*}
    \|v-P_h v\|^2_{L^2 (\Omega)}
    = \int_{\Omega} \nabla (w-I_hw) \cdot \nabla (v-P_h v) \,dx
    \leq |w-I_hw|_{H^1 (\Omega)} |v-P_h v|_{H^1 (\Omega)}.
\end{equation*}
Invoking \cref{Thm:SZ_averaged_H1}, we get
\begin{equation*}
\begin{split}
    \|v-P_h v\|^2_{L^2 (\Omega)}
    &\lesssim
    \rho^{\frac{3}{2}} \sigma^{-\frac{3}{2}} d^2 h
    | w |_{H^2 (\Omega)} |v-P_h v|_{H^1 (\Omega)} \\
    &\stackrel{\eqref{P_h}}{\leq}
    \rho^{\frac{3}{2}} \sigma^{-\frac{3}{2}} d^2 h
    |w|_{H^2 (\Omega)} |v|_{H^1 (\Omega)} \\
    &\stackrel{\eqref{Thm2:P_h}}{\leq}
    \rho^{\frac{3}{2}} \sigma^{-\frac{3}{2}} d^2 h
    \|v-P_h v\|_{L^2 (\Omega)} |v|_{H^1 (\Omega)}.
\end{split}
\end{equation*}
Dividing both sides by $\|v-P_h v\|_{L^2 (\Omega)}$ yields the desired result.
\end{proof}

\section{Bramble--Pasciak--Xu preconditioners} 
\label{Sec:BPX}
In this section, we present our main result: we show that the condition numbers of BPX preconditioners~\cite{BPX:1990,Xu:1989} depend on the spatial dimension \(d\) only polynomially.
We begin by analyzing the dimension dependence of the strengthened Cauchy--Schwarz inequalities~\cite{Xu:1992} and the norm equivalence theorem~\cite{BY:1993,DK:1992,Oswald:1990,XQ:1994}, two important technical tools in the analysis of multilevel methods.
Using the norm equivalence theorem, we then derive the BPX preconditioner.
Finally, by invoking the theory of parallel subspace correction methods~\cite{PX:2025,Xu:1992,XZ:2002}, we obtain an improved estimate for the condition number of the BPX preconditioner.

\subsection{Multilevel finite element spaces}
A standard assumption on multilevel finite element spaces, commonly used in multilevel methods~(see, e.g.,~\cite{Xu:1992}), is stated in \cref{Ass:multilevel}.

\begin{assumption}
\label{Ass:multilevel}
For $J \ge 1$, the triangulation $\mathcal{T}_h$ admits a sequence of triangulations
\(\mathcal{T}_1, \dots, \mathcal{T}_J = \mathcal{T}_h\)
satisfying the following conditions:
\begin{itemize}
    \item $\mathcal{T}_l$ is a refinement of $\mathcal{T}_{l-1}$ for \(2 \le l \le J\);
    \item for each \(1 \le l \le J\), the characteristic mesh size \(h_l\) of $\mathcal{T}_l$ satisfies
    \begin{equation*}
        h_l \eqsim \gamma^{2l}
    \end{equation*}
    for some $\gamma \in (0,1)$.
    We set $h_0 \eqsim 1$.
\end{itemize}
\end{assumption}

Note that the number of levels $J$ satisfies \(J = \mathcal{O}(|\log h|)\).
For each triangulation $\mathcal{T}_l$, we define the corresponding finite element space
\[
    V_l := \{ v \in H_0^1(\Omega) : v|_\tau \in \mathbb{P}_1(\tau)
          \text{ for all } \tau \in \mathcal{T}_l \},
          \quad 1 \leq l \leq J.
\]
These spaces form a nested sequence of subspaces of $V_{h,0}$:
\begin{equation*}
    V_1 \subset V_2 \subset \cdots \subset V_J = V_{h,0},
\end{equation*}
which yields the following multilevel space decomposition:
\begin{equation}
\label{multilevel_decomposition}
    V_{h,0} = \sum_{l=1}^J V_l.
\end{equation}

The overlap among the subspaces in the multilevel decomposition~\eqref{multilevel_decomposition} is essential: although it introduces redundancy in representation, this redundancy enables multilevel methods to efficiently reduce errors across different frequency ranges.
For algebraic perspectives explaining how such redundancy improves the convergence rate of iterative methods, see~\cite{PX:2025}.

Let \(Q_l \colon V_{h,0} \to V_l\) be the \(L^2\)-orthogonal projection onto \(V_l\), with the convention \(Q_0 = 0\).
Similarly, let
\(P_l \colon H_0^1(\Omega) \to V_l\) be the \(H^1\)-orthogonal projection
onto \(V_l\), i.e.,
\[
    \int_\Omega \nabla(P_l v)\cdot \nabla w_l\,dx
    =
    \int_\Omega \nabla v \cdot \nabla w_l\,dx,
    \quad w_l \in V_l,
\]
with the convention \(P_0 = 0\). An important property is that, on the range \(\mathcal{R}(Q_l - Q_{l-1})\) of the difference operator \(Q_l - Q_{l-1}\), the \(L^2\)-norm \(\|\cdot\|_{L^2(\Omega)}\) and the scaled \(H^1\)-seminorm \(h_l\, |\cdot|_{H^1(\Omega)}\) are equivalent, as summarized in \cref{Lem:W_l}.

\begin{lemma}
\label{Lem:W_l}
Suppose that \cref{Ass:multilevel} holds and that \cref{Ass:triangulation} holds for $\mathcal{T}_l$, $1\leq l \leq J$.
For any $1\leq l \leq J$, we have
\begin{equation*}
\rho^{-\frac{1}{2}} \sigma^{\frac{1}{2}} d^{-2} \gamma^2 h_l^{-1} \| v \|_{L^2 (\Omega)}
\lesssim | v |_{H^1 (\Omega)}
\lesssim \rho \sigma^{-1} d^{\frac{3}{2}} h_l^{-1} \| v \|_{L^2 (\Omega)},
\quad v \in \mathcal{R}(Q_l - Q_{l-1}).
\end{equation*}
where $Q_l \colon V_{h,0} \to V_l$ denotes the $L^2$-orthogonal projection onto $V_l$, and $Q_0 = 0$.
\end{lemma}
\begin{proof}
Since the right inequality is clear from the global inverse inequality (\cref{Cor:inverse}), we only prove the left inequality.
Take any $v \in \mathcal{R}(Q_l - Q_{l-1})$.
There exists $w \in V_{h,0}$ such that
\begin{equation}
\label{Lem1:W_l}
    v = (Q_l - Q_{l-1}) w = (I - Q_{l-1}) (Q_l - Q_{l-1}) w.
\end{equation}
Then by the approximation property of the $ L^2 $-orthogonal projection~(\cref{Thm:Q_h}) for $l\geq 2$ and the Friedrichs inequality~(\cref{Lem:Friedrichs}) for $l=1$, we have
\begin{multline*}
    \| v \|_{L^2 (\Omega)}
    \stackrel{\eqref{Lem1:W_l}}{=} \|(I-Q_{l-1}) (Q_{l}-Q_{l-1}) w \|_{L^2 (\Omega) } \\
    \lesssim \rho^{\frac{1}{2}} \sigma^{-\frac{1}{2}} d^2 \gamma^{-2} h_l | ( Q_{l} - Q_{l-1} ) w |_{H^1 (\Omega)}
    \stackrel{\eqref{Lem1:W_l}}{=} \rho^{\frac{1}{2}} \sigma^{-\frac{1}{2}} d^2 \gamma^{-2} h_l | v |_{H^1 (\Omega)}.
\end{multline*}
This completes the proof.
\end{proof}


\subsection{Strengthened Cauchy--Schwarz inequalities}
The strengthened Cauchy--Schwarz inequalities are important technical tools for analyzing the multilevel space decomposition~\eqref{multilevel_decomposition}.
Here, we closely follow the arguments in~\cite[Section~6.1]{Xu:1992}, with careful attention to the dependence on the spatial dimension \(d\) and relevant parameters.

In \cref{Lem:SCS_V_l}, we reproduce the Cauchy--Schwarz-type inequality from~\cite[Lemma~6.1]{Xu:1992}, now with explicit dependence on the dimension and parameters.

\begin{lemma}
\label{Lem:SCS_V_l}
Suppose that \cref{Ass:multilevel} holds and that \cref{Ass:triangulation} holds for $\mathcal{T}_l$, $1\leq l \leq J$.
For any  $1 \leq l \leq k \leq J$, we have
\begin{equation*}
\int_{\Omega} \nabla v \cdot \nabla w \,dx \lesssim  \rho \sigma^{-1} d^{\frac{3}{2}} \gamma^{k-l} h_k^{-1} | v |_{H^1 (\Omega)} \| w \|_{L^2 (\Omega)},
\quad v \in V_l,\ w \in V_k.
\end{equation*}
\end{lemma}
\begin{proof}
Take any $\tau \in \mathcal{T}_l$.
Since $\nabla v$ is constant on $\tau$, by the shape-regularity and the quasi-uniformity~(see \cref{Ass:triangulation}), we have
\begin{equation}
\label{Lem1:SCS_V_l}
\frac{\| \nabla v \|_{L^2 (\partial \tau)}^2}{\| \nabla v \|_{L^2 (\tau)}^2}
= \frac{| \partial \tau |}{| \tau |}
= d r_{\tau}^{-1}
\leq \rho d h_{\tau}^{-1}
\leq \rho \sigma^{-1} d h_l^{-1}.
\end{equation}
It follows from the Green's identity that
\begin{equation}
\label{Lem2:SCS_V_l}
\begin{split}
\int_{\tau} \nabla v \cdot \nabla w \,dx
&= \int_{\partial \tau} w \frac{\partial v}{\partial n} \,ds \\
&\leq \| \nabla v \|_{L^2 ( \partial \tau )} \| w \|_{L^2 ( \partial \tau )} \\
&\lesssim \rho \sigma^{-1} d^{\frac{3}{2}} h_l^{-1/2} h_k^{-1/2} \| \nabla v \|_{L^2 (\tau)}\| w \|_{L^2 (\tau)} \\
&\eqsim \rho \sigma^{-1} d^{\frac{3}{2}} \gamma^{k-l} h_k^{-1} \| \nabla v \|_{L^2 (\tau)}\| w \|_{L^2 (\tau)},
\end{split}   
\end{equation}
where the penultimate step follows from the discrete trace inequality (\cref{Lem:trace_discrete}) together with \eqref{Lem1:SCS_V_l}, and the final $\eqsim$ uses \cref{Ass:multilevel}.
Summing~\eqref{Lem2:SCS_V_l} over all $\tau \in \mathcal{T}_l$ followed by applying the Cauchy--Schwarz inequality yields
\begin{equation*}
\begin{split}
\int_{\Omega} \nabla v \cdot \nabla w \,dx
&\lesssim \rho \sigma^{-1} d^{\frac{3}{2}} \gamma^{k-l} h_k^{-1} \sum_{\tau \in \mathcal{T}_l} | v |_{H^1 (\tau)}\| w \|_{L^2 (\tau)} \\
&\leq \rho \sigma^{-1} d^{\frac{3}{2}} \gamma^{k-l} h_k^{-1} \left(\sum_{\tau\in\mathcal{T}_l}
| v |_{H^1 (\tau)}^2 \right)^{\frac{1}{2}} \left( \sum_{\tau \in \mathcal{T}_l} \| w \|_{L^2 (\tau) }^2 \right)^{\frac{1}{2}} \\
&= \rho \sigma^{-1} d^{\frac{3}{2}} \gamma^{k-l} h_k^{-1} | v |_{H^1 (\Omega)} \| w \|_{L^2 (\Omega)}, 
\end{split}
\end{equation*}
which completes the proof.
\end{proof}

We now state the strengthened Cauchy--Schwarz inequalities among the spaces \(\{ \mathcal{R}(Q_l - Q_{l-1}) \}_{l=1}^J\) in \cref{Lem:SCS_W_l} (cf.~\cite[Lemma~6.2]{Xu:1992}).
This result follows directly from \cref{Lem:W_l,Lem:SCS_V_l}.

\begin{lemma}[Strengthened Cauchy--Schwarz inequalities]
\label{Lem:SCS_W_l}
Suppose that \cref{Ass:multilevel} holds and that \cref{Ass:triangulation} holds for $\mathcal{T}_l$, $1\leq l \leq J$.
For any $1\leq l \leq k \leq J$, we have
\begin{multline*}
\int_{\Omega} \nabla v \cdot \nabla w \,dx
\lesssim
\rho^{\frac{3}{2}} \sigma^{-\frac{3}{2}} d^{\frac{7}{2}}
\gamma^{k-l-2} | v |_{H^1 (\Omega)} | w |_{H^1 (\Omega)}, \\
v \in \mathcal{R} (Q_l - Q_{l-1}),\ w \in \mathcal{R} (Q_k - Q_{k-1}),
\end{multline*}
where $Q_l \colon V_{h,0} \to V_l$ denotes the $L^2$-orthogonal projection onto $V_l$, and $Q_0 = 0$.
\end{lemma}

\subsection{Norm equivalence theorem}
We now present the norm equivalence theorem~\cite{BY:1993,DK:1992,Oswald:1990,XQ:1994}, one of the most important theoretical results in the analysis of multilevel iterative methods, in \cref{Thm:norm_equivalence}.
We note that although we assume full elliptic regularity for convenience (see \cref{Ass:triangulation}), the norm equivalence theorem itself does not require full elliptic regularity and remains valid for a broader class of domains.

\begin{theorem}
\label{Thm:norm_equivalence} 
Suppose that \cref{Ass:multilevel} holds and that \cref{Ass:triangulation} holds for $\mathcal{T}_l$, $1\leq l \leq J$.
Then we have
\begin{equation*}
     \underline{C}_{\rho, \sigma, d, \gamma}^{\mathrm{NE}} \sum_{l=1}^J | (Q_l - Q_{l-1}) v |_{H^1 (\Omega)}^2 
    \lesssim  | v |_{H^1 (\Omega)}^2 
    \lesssim  \overline{C}_{\rho, \sigma, d, \gamma}^{\mathrm{NE}} \sum_{l=1}^{J} | (Q_l - Q_{l-1}) v |_{H^1 (\Omega)}^2,
    \text{ } v \in V_{h,0},
\end{equation*}
where $Q_l \colon V_{h,0} \to V_l$ denotes the $L^2$-orthogonal projection onto $V_l$, $Q_0 = 0$, and
\begin{equation*}
    \underline{C}_{\rho, \sigma, d, \gamma}^{\mathrm{NE}}
    =
    \rho^{-5} \sigma^5 d^{-7} \gamma^4 (1-\gamma^2) (1-\gamma^4),
    \quad
    \overline{C}_{\rho, \sigma, d, \gamma}^{\mathrm{NE}}
    =
    \frac{\rho^{\frac{3}{2}} \sigma^{-\frac{3}{2}} d^{\frac{7}{2}}}{ \gamma^{2} (1-\gamma)}.
\end{equation*}
\end{theorem}
\begin{proof}
For each $1 \leq l \leq J$, we define
\begin{equation}
\label{v_l}
v_l := (P_l - P_{l-1}) v = (I - P_{l-1}) (P_l - P_{l-1}) v \in V_l,
\end{equation}
so that we have
\begin{equation}
\label{v_l_decomposition}
v = \sum_{l=1}^{J} v_l, \quad
| v |_{H^1 (\Omega)}^2 = \sum_{l=1}^J | v_l |_{H^1 (\Omega)}^2.
\end{equation}
We readily observe that 
\begin{equation}
\label{Thm2:norm_equivalence}
(Q_k - Q_{k-1}) v_l = 0, \quad v_l \in V_l,\ k > l.
\end{equation}
Moreover, for $1\leq k\leq l$, we have
\begin{equation}
\label{Thm3:norm_equivalence}
\begin{split}
    \|(Q_k - Q_{k-1}) v_l\|_{L^2 (\Omega)}
    &\leq \| v_l\|_{L^2 (\Omega)} \\
    &\stackrel{\eqref{v_l}}{=} \|(I-P_{l-1})(P_l-P_{l-1})v\|_{L^2 (\Omega)} \\
    &\lesssim
    \rho^{\frac{3}{2}} \sigma^{-\frac{3}{2}} d^2
    h_{l-1} | (P_l - P_{l-1}) v|_{H^1 (\Omega)} \\
    &\stackrel{\eqref{v_l}}{=}
    \rho^{\frac{3}{2}} \sigma^{-\frac{3}{2}} d^2
    h_{l-1} | v_l |_{H^1 (\Omega)}.
\end{split}
\end{equation}
where the penultimate inequality is due to \cref{Thm:P_h} for $l\geq 2$ and the Friedrichs inequality~(\cref{Lem:Friedrichs}) for $l=1$.
From~\eqref{v_l_decomposition} and~\eqref{Thm2:norm_equivalence}, we get
\begin{align}
\sum_{l=1}^J| (Q_l - Q_{l-1}) v|_{H^1 (\Omega)}^2
&=\sum_{i,j=1}^{J}\sum_{l=1}^{\min (i, j)}
\int_{\Omega} \nabla ( (Q_l - Q_{l-1}) v_i ) \cdot \nabla ( (Q_l - Q_{l-1}) v_j ) \,dx \notag \\
&\leq \sum_{i,j=1}^J\sum_{l=1}^{\min (i, j)} | (Q_l - Q_{l-1}) v_i|_{H^1 (\Omega)}| (Q_l - Q_{l-1}) v_j|_{H^1 (\Omega)}. \label{Thm4:norm_equivalence}
\end{align}
By \cref{Lem:W_l} and~\eqref{Thm3:norm_equivalence}, we obtain
\begin{equation*}
\begin{split}
\sum_{i,j=1}^J &\sum_{l=1}^{\min (i, j)} | (Q_l - Q_{l-1}) v_i|_{H^1 (\Omega)}| (Q_l - Q_{l-1}) v_j|_{H^1 (\Omega)} \\
&\lesssim \rho^2 \sigma^{-2} d^3
\sum_{i,j=1}^J\sum_{l=1}^{\min (i, j)}
h_l^{-2} \| (Q_l - Q_{l-1}) v_i\|_{L^2 (\Omega)}\| (Q_l - Q_{l-1}) v_j\|_{L^2 (\Omega)} \\
&\lesssim
\rho^5 \sigma^{-5} d^7
\sum_{i,j=1}^J \sum_{l=1}^{\min (i, j)}
h_l^{-2} h_{i-1}h_{j-1} |v_i|_{H^1 (\Omega)}|v_j|_{H^1 (\Omega)}.
\end{split}
\end{equation*}
Using \cref{Ass:multilevel} and the inequality
\begin{equation*}
    \sum_{l=1}^{\min(i,j)} h_l^{-2}
    \eqsim \sum_{l=1}^{\min(i,j)} \gamma^{-4l}
    = \frac{\gamma^{-4 \min(i,j)} - 1}{1 - \gamma^4}
    \lesssim \frac{h_{\min (i,j)}^{-2}}{1-\gamma^{4}},
\end{equation*}
we get
\begin{equation*}
\begin{split}
\sum_{i,j=1}^J &\sum_{l=1}^{\min (i, j)} h_l^{-2} h_{i-1}h_{j-1} |v_i|_{H^1 (\Omega)}|v_j|_{H^1 (\Omega)} \\
&\lesssim \frac{1}{\gamma^4 (1- \gamma^4)} \sum_{i,j=1}^J h_{\min (i, j)}^{-2} h_i h_j |v_i|_{H^1 (\Omega)}|v_j|_{H^1 (\Omega)} \\
&\eqsim \frac{1}{\gamma^4 (1-\gamma^4)} \sum_{i,j=1}^J \gamma^{2|i-j|}|v_i|_{H^1 (\Omega)}|v_j|_{H^1 (\Omega)}.
\end{split}
\end{equation*}
With elementary manipulations, we deduce
\begin{equation}
\label{Thm7:norm_equivalence}
\begin{split}
\sum_{i,j=1}^J \gamma^{2|i-j|}|v_i|_{H^1 (\Omega)}|v_j|_{H^1 (\Omega)} 
&\lesssim \sum_{i,j=1}^J \gamma^{2|i-j|} ( |v_i|_{H^1 (\Omega)}^2+|v_j|_{H^1 (\Omega)}^2 ) \\
&\eqsim \sum_{i=1}^J \left( \sum_{j=1}^J\gamma^{2|i-j|} \right) |v_i|_{H^1 (\Omega)}^2 \\
&\lesssim \frac{1}{1 - \gamma^2}\sum_{i=1}^J |v_i|_{H^1 (\Omega)}^2 \\
&\stackrel{\eqref{v_l_decomposition}}{=} \frac{1}{1 - \gamma^2} | v |_{H^1 (\Omega)}^2.
\end{split}    
\end{equation}
Combining \eqref{Thm4:norm_equivalence}--\eqref{Thm7:norm_equivalence} yields the desired estimate for
\(\underline{C}_{\rho,\sigma,d,\gamma}^{\mathrm{NE}} \).

On the other hand, we use \cref{Lem:SCS_W_l} to obtain 
\begin{equation*}
\begin{split}
|v|_{H^1 (\Omega)}^2
&= \sum_{i,j=1}^J \int_{\Omega} \nabla ( (Q_i - Q_{i-1}) v ) \cdot \nabla ( (Q_j - Q_{j-1}) v ) \,dx \\
&\lesssim
\rho^{\frac{3}{2}} \sigma^{-\frac{3}{2}} d^{\frac{7}{2}}
\sum_{i,j=1}^J \gamma^{|i-j|-2}
| (Q_i - Q_{i-1}) v |_{H^1 (\Omega)}
| (Q_j - Q_{j-1}) v |_{H^1 (\Omega)} \\
&\leq
\frac{\rho^{\frac{3}{2}} \sigma^{-\frac{3}{2}} d^{\frac{7}{2}} }{\gamma^2 (1-\gamma)}
\sum_{l=1}^J| (Q_l - Q_{l-1}) v|_{H^1 (\Omega)}^2.
\end{split}
\end{equation*}
This completes the proof of the estimate for \(\overline{C}_{\rho,\sigma,d,\gamma}^{\mathrm{NE}} \).
\end{proof}

Combining \cref{Lem:W_l} and \cref{Thm:norm_equivalence}, we obtain \cref{Cor:norm_equivalence}, which plays an important role in the derivation of BPX preconditioners.

\begin{corollary}
\label{Cor:norm_equivalence}
Suppose that \cref{Ass:multilevel} holds and that \cref{Ass:triangulation} holds for $\mathcal{T}_l$, $1\leq l \leq J$.
Then we have
\begin{equation*}
    \underline{C}_{\rho, \sigma, d, \gamma}^{\hat{A}} (\hat{A} v, v)_{L^2 (\Omega)} 
\lesssim | v |_{H^1 (\Omega)}^2 
\lesssim \overline{C}_{\rho, \sigma, d, \gamma}^{\hat{A}} (\hat{A} v, v)_{L^2 (\Omega)} , 
\quad v \in V_{h,0},
\end{equation*}
where the linear operator $\hat{A} \colon V_{h,0} \to V_{h,0}$ is defined by
\begin{equation*}
    \hat{A} = \sum_{l=1}^J h_l^{-2} (Q_l - Q_{l-1}),
\end{equation*}
and $Q_l \colon V_{h,0} \to V_l$ denotes the $L^2$-orthogonal projection onto $V_l$, $Q_0 = 0$, and
\begin{equation*}
    \underline{C}_{\rho, \sigma, d, \gamma}^{\hat{A}}
    =
    \rho^{-6} \sigma^6 d^{-11} \gamma^8 (1-\gamma^2) (1-\gamma^4),
    \quad
    \overline{C}_{\rho, \sigma, d, \gamma}^{\hat{A}}
    =
    \frac{\rho^{\frac{7}{2}} \sigma^{-\frac{7}{2}} d^{\frac{13}{2}}}{ \gamma^{2} (1-\gamma) }.
\end{equation*}
\end{corollary}

\subsection{Construction of BPX preconditioners}
While the BPX preconditioner was originally developed as part of an effort to parallelize classical multigrid methods~\cite{BPX:1990,Xu:1989}, here we derive it in an alternative way, namely through the norm equivalence theorem established above.

Our starting point is the linear operator $\hat{A}$ defined in \cref{Cor:norm_equivalence}.
By \cref{Ass:multilevel}, we have
\[
    \hat{A}^{-1}
    = \sum_{l=1}^J h_l^2 (Q_l - Q_{l-1})
    \eqsim \sum_{l=1}^J \gamma^{4l} (Q_l - Q_{l-1}).
\]
Hence, for any $v \in V_{h,0}$, we have
\begin{equation}
\label{BPX_derivation1}
\begin{split}
    (\hat{A}^{-1} v, v)_{L^2(\Omega)}
    &\eqsim \sum_{l=1}^J \gamma^{4l} \big( (Q_l - Q_{l-1}) v, v \big)_{L^2(\Omega)} \\
    &= \gamma^{4J} (v,v)_{L^2(\Omega)}
       + \sum_{l=1}^{J-1} \gamma^{4l} (1 - \gamma^{4})
         (Q_l v, v)_{L^2(\Omega)} \\
    &\eqsim h_J^2 (v,v)_{L^2(\Omega)}
       + (1 - \gamma^{4}) \sum_{l=1}^{J-1} h_l^2 (Q_l v, v)_{L^2(\Omega)}.
\end{split}
\end{equation}

We now consider the following instance of BPX preconditioners (see~\cite{BPX:1990}):
\begin{equation}
\label{BPX}
    B = \sum_{l=1}^J h_l^2 Q_l.
\end{equation}
Then~\eqref{BPX_derivation1} implies
\[
    (1 - \gamma^{4})\, (Bv, v)_{L^2(\Omega)}
    \lesssim (\hat{A}^{-1} v, v)_{L^2(\Omega)}
    \lesssim (Bv, v)_{L^2(\Omega)},
    \quad v \in V_{h,0}.
\]
Equivalently,
\begin{equation}
\label{BPX_derivation2}
    (1 - \gamma^{4})\, (\hat{A} v, v)_{L^2(\Omega)}
    \lesssim (B^{-1} v, v)_{L^2(\Omega)}
    \lesssim (\hat{A} v, v)_{L^2(\Omega)},
    \quad v \in V_{h,0}.
\end{equation}

Combining~\eqref{BPX_derivation2} with \cref{Cor:norm_equivalence}, we obtain
\begin{equation}
\label{BPX_derivation3}
    (1 - \gamma^{4})\, (\overline{C}_{\rho,\sigma,d,\gamma}^{\hat{A}})^{-1}
    |v|_{H^1(\Omega)}^2
    \lesssim (B^{-1} v, v)_{L^2(\Omega)}
    \lesssim (\underline{C}_{\rho,\sigma,d,\gamma}^{\hat{A}})^{-1}
    |v|_{H^1(\Omega)}^2,
    \quad v \in V_{h,0},
\end{equation}
where $\underline{C}_{\rho,\sigma,d,\gamma}^{\hat{A}}$ and $\overline{C}_{\rho,\sigma,d,\gamma}^{\hat{A}}$ are the constants in \cref{Cor:norm_equivalence}.

The inequality~\eqref{BPX_derivation3} shows that the spectral condition number $\kappa(BA)$ given by
\begin{equation*} 
\kappa(BA) = 
\frac{\displaystyle \sup_{v\in V_{h,0}\setminus\{0\}} \frac{(Av,v)_{L^2(\Omega)}}{(B^{-1}v,v)_{L^2(\Omega)}} }
{\displaystyle \inf_{v\in V_{h,0}\setminus\{0\}} \frac{(Av,v)_{L^2(\Omega)}}{(B^{-1}v,v)_{L^2(\Omega)}} }, 
\end{equation*}
where $A \colon V_{h,0} \to V_{h,0}$ is defined by
\begin{equation}
\label{A}
    (Av, w)_{L^2(\Omega)}
    = \int_{\Omega} \nabla v \cdot \nabla w \, dx,
    \quad v, w \in V_{h,0},
\end{equation}
is bounded independently of the mesh size \(h\) and the number of levels \(J\), and depends on the dimension \(d\) and the shape-regularity parameter \(\rho\) only polynomially.
This is precisely the main goal of this paper.

On the other hand, by using the theory of parallel subspace correction methods~\cite{Xu:1992,XZ:2002}, we can obtain an even sharper bound on \(\kappa(BA)\), as we discuss next.

\subsection{BPX preconditioners as parallel subspace correction methods}
As discussed in~\cite{Xu:1992}, the BPX preconditioner~\eqref{BPX}, or more general forms with general smoothers~(see~\cite{BPX:1990}) are interpreted as parallel subspace correction methods based on the multilevel space decomposition~\eqref{multilevel_decomposition}.

From the well-established convergence theory of parallel subspace correction methods~(see, e.g.,~\cite[Theorem~6.2]{PX:2025} and~\cite[Lemma~2.4]{XZ:2002}), we have the following result.

\begin{lemma}
\label{Lem:PSC}
The Bramble--Pasciak--Xu preconditioner $B$ defined in~\eqref{BPX} satisfies
\begin{equation*}
    (B^{-1} v, v)_{L^2 (\Omega)} = \inf_{v_l \in V_l,\ \sum_{l=1}^J v_l = v} \sum_{l=1}^J h_l^{-2} \| v_l \|_{L^2 (\Omega)}^2,
    \quad v \in V_{h,0}.
\end{equation*}
\end{lemma}

Using \cref{Lem:PSC}, we obtain the following estimate, which is sharper than the constant \( (\underline{C}_{\rho, \sigma, d, \gamma}^{\hat{A}})^{-1}\) appearing in~\eqref{BPX_derivation3}.

\begin{lemma}
\label{Lem:BPX_min}
Suppose that \cref{Ass:multilevel} holds and that \cref{Ass:triangulation} holds for $\mathcal{T}_l$, $1\leq l \leq J$.
Then the Bramble--Pasciak--Xu preconditioner $B$ defined in~\eqref{BPX} satisfies
\begin{equation*}
    (B^{-1} v, v)_{L^2 (\Omega)}
    \lesssim
    \rho^3 \sigma^{-3} d^4 \gamma^{-4} | v |_{H^1 (\Omega)}^2,
    \quad v \in V_{h,0}.
\end{equation*}
\end{lemma}
\begin{proof}
For any $1\leq l\leq J$, we define $v_l \in V_l$ as in~\eqref{v_l}.
By \cref{Lem:PSC},~\eqref{Thm3:norm_equivalence}, \cref{Ass:multilevel}, and~\eqref{v_l_decomposition}, we get
\begin{equation*}
\begin{split}
    (B^{-1} v, v)_{L^2 (\Omega)}
    &\leq \sum_{l=1}^J h_l^{-2} \| v_l \|_{L^2 (\Omega)}^2 \\
    &\lesssim \rho^3 \sigma^{-3} d^4
    \sum_{l=1}^J h_l^{-2} h_{l-1}^2 | v_l |_{H^1 (\Omega)}^2 \\
    &\eqsim \rho^3 \sigma^{-3}d^4 \gamma^{-4}
    \sum_{l=1}^J | v_l |_{H^1 (\Omega)}^2 \\
    &= \rho^3 \sigma^{-3}d^4 \gamma^{-4} | v |_{H^1 (\Omega)}^2.
\end{split}
\end{equation*}
This completes the proof.
\end{proof}

In addition, \cref{Lem:BPX_max} provides a lower bound for \((B^{-1}v,v)_{L^2(\Omega)}\) that is sharper than the one obtained from~\eqref{BPX_derivation3}.

\begin{lemma}
\label{Lem:BPX_max}
Suppose that \cref{Ass:multilevel} holds and that \cref{Ass:triangulation} holds for $\mathcal{T}_l$, $1\leq l \leq J$.
Then the Bramble--Pasciak--Xu preconditioner $B$ defined in~\eqref{BPX} satisfies
\begin{equation*}
    (B^{-1}v,v)_{L^2(\Omega)}
    \gtrsim
    \rho^{-2}\sigma^2 d^{-3}(1-\gamma)
    |v|_{H^1(\Omega)}^2,
    \quad v\in V_{h,0}.
\end{equation*}
\end{lemma}
\begin{proof}
Take any $v\in V_{h,0}$ and any decomposition
\[
    v=\sum_{l=1}^J v_l,\quad v_l\in V_l .
\]
Then
\begin{equation}
\label{Lem1:BPX_max}
\begin{split}
    |v|_{H^1(\Omega)}^2
    &=
    \sum_{l=1}^J |v_l|_{H^1(\Omega)}^2
    +
    2\sum_{1\leq l<k\leq J}
    \int_\Omega \nabla v_l\cdot\nabla v_k\,dx .
\end{split}
\end{equation}
By \cref{Cor:inverse}, we have
\begin{equation}
\label{Lem2:BPX_max}
    |v_l|_{H^1(\Omega)}^2
    \lesssim
    \rho^2\sigma^{-2}d^3 h_l^{-2}
    \|v_l\|_{L^2(\Omega)}^2 .
\end{equation}
Moreover, by \cref{Lem:SCS_V_l} and \cref{Cor:inverse}, for $l<k$,
\begin{equation}
\label{Lem3:BPX_max}
\begin{split}
    \left|
    \int_\Omega \nabla v_l\cdot\nabla v_k\,dx
    \right|
    &\lesssim
    \rho\sigma^{-1}d^{\frac32}\gamma^{k-l}h_k^{-1}
    |v_l|_{H^1(\Omega)}\|v_k\|_{L^2(\Omega)} \\
    &\lesssim
    \rho^2\sigma^{-2}d^3
    \gamma^{k-l}
    h_l^{-1}h_k^{-1}
    \|v_l\|_{L^2(\Omega)}
    \|v_k\|_{L^2(\Omega)} .
\end{split}
\end{equation}
Combining~\eqref{Lem1:BPX_max},~\eqref{Lem2:BPX_max}, and~\eqref{Lem3:BPX_max}, we get
\begin{equation}
\label{Lem4:BPX_max}
\resizebox{\textwidth}{!}{$\displaystyle
    |v|_{H^1(\Omega)}^2
    \lesssim
    \rho^2\sigma^{-2}d^3
    \left(
    \sum_{l=1}^J h_l^{-2}\|v_l\|_{L^2(\Omega)}^2
    +
    \sum_{1\leq l<k\leq J}
    \gamma^{k-l}
    h_l^{-1}h_k^{-1}
    \|v_l\|_{L^2(\Omega)}
    \|v_k\|_{L^2(\Omega)}
    \right).
$}
\end{equation}
For the second term in the right-hand side of~\eqref{Lem4:BPX_max}, by an elementary inequality
$2xy\leq x^2+y^2$, we have
\begin{equation}
\label{Lem5:BPX_max}
\begin{split}
    &\sum_{1\leq l<k\leq J}
    \gamma^{k-l}
    h_l^{-1}h_k^{-1}
    \|v_l\|_{L^2(\Omega)}
    \|v_k\|_{L^2(\Omega)} \\
    &\leq
    \frac{1}{2}
    \sum_{1\leq l<k\leq J}
    \gamma^{k-l}
    \left(
    h_l^{-2}\|v_l\|_{L^2(\Omega)}^2
    +
    h_k^{-2}\|v_k\|_{L^2(\Omega)}^2
    \right) \\
    &=
    \frac{1}{2}
    \sum_{l=1}^J
    h_l^{-2}\|v_l\|_{L^2(\Omega)}^2
    \sum_{k=l+1}^J \gamma^{k-l}
    +
    \frac{1}{2}
    \sum_{k=1}^J
    h_k^{-2}\|v_k\|_{L^2(\Omega)}^2
    \sum_{l=1}^{k-1} \gamma^{k-l} \\
    &\leq
    \frac{\gamma}{1-\gamma}
    \sum_{l=1}^J
    h_l^{-2}\|v_l\|_{L^2(\Omega)}^2 .
\end{split}
\end{equation}
Combining~\eqref{Lem4:BPX_max} and~\eqref{Lem5:BPX_max} yields
\begin{equation*}
    |v|_{H^1(\Omega)}^2
    \lesssim
    \frac{\rho^2\sigma^{-2}d^3}{1-\gamma}
    \sum_{l=1}^J h_l^{-2}\|v_l\|_{L^2(\Omega)}^2 .
\end{equation*}
Since the decomposition \(v=\sum_{l=1}^J v_l\) was arbitrary, taking the infimum
over all such decompositions and using \cref{Lem:PSC} gives
\[
    |v|_{H^1(\Omega)}^2
    \lesssim
    \frac{\rho^2\sigma^{-2}d^3}{1-\gamma}
    (B^{-1}v,v)_{L^2(\Omega)}.
\]
This completes the proof.
\end{proof}

Combining~\cref{Lem:BPX_max,Lem:BPX_min}, we finally obtain the following estimate.

\begin{theorem}
\label{Thm:BPX}
Suppose that \cref{Ass:multilevel} holds and that \cref{Ass:triangulation} holds for $\mathcal{T}_l$, $1\leq l \leq J$.
Then the Bramble--Pasciak--Xu preconditioner $B$ defined in~\eqref{BPX} satisfies
\begin{equation*}
    \rho^{-2}\sigma^2 d^{-3}(1-\gamma)
    | v |_{H^1 (\Omega)}^2
    \lesssim (B^{-1} v, v)_{L^2 (\Omega)}
    \lesssim
    \rho^3 \sigma^{-3} d^4 \gamma^{-4}
    | v |_{H^1 (\Omega)}^2,
    \quad v \in V_{h,0}.
\end{equation*}
Consequently, we have
\begin{equation*}
    \kappa (BA)
    \lesssim
    \frac{\rho^5 \sigma^{-5} d^7}{\gamma^4 (1-\gamma)},
\end{equation*}
where the operator $A$ was defined in~\eqref{A}.
\end{theorem}

In the special case of multilevel uniform Freudenthal triangulations, we have \(\sigma=1\) and \(\rho \eqsim d^{3/2}\), as discussed in \cref{Sec:Triangulations}.
Therefore, \cref{Thm:BPX} gives
\[
    \kappa(BA)
    \lesssim
    \frac{d^{\frac{29}{2}}}{\gamma^4(1-\gamma)}.
\]
Thus, the BPX-preconditioned system remains polynomially conditioned in \(d\).

\subsection{Jacobi and Richardson smoothers}

In addition to the preconditioner $B$ defined in~\eqref{BPX}, several variants of BPX preconditioners have been used in the literature~\cite{BPX:1990,Xu:1989,XQ:1994}.
We consider the following Jacobi and Richardson variants:
\begin{equation}
\label{BPX_variants}
B_{\mathrm{Jac}}v
=
\sum_{l=1}^J h_l^2 \sum_{i=1}^{N_l} Q_i^l v,
\quad
B_{\mathrm{Ric}}v
=
\sum_{l=1}^J h_l^{2-d}
\sum_{i=1}^{N_l}
(v,\phi_i^l)_{L^2(\Omega)}\phi_i^l,
\quad v\in V_{h,0}.
\end{equation}
Here, for $1 \leq l \leq J$, $N_l$ is the number of interior nodes of $\mathcal{T}_l$, $\{ \phi_i^l \}_{i=1}^{N_l}$ is the nodal basis for $V_l$, and $Q_i^l$ is the $L^2$-orthogonal projection onto $\operatorname{span}\{\phi_i^l\}$.

Let $\mathsf{M}_l$ be the mass matrix on $V_l$ and let $\mathsf{D}_l$ be its diagonal part:
\[
    \mathsf{M}_l
    =
    \left[
        (\phi_i^l,\phi_j^l)_{L^2(\Omega)}
    \right]_{i,j=1}^{N_l},
    \quad
    \mathsf{D}_l
    =
    \operatorname{diag}(\mathsf{M}_l).
\]
These variants can be obtained from $B$ by replacing the level mass matrix $\mathsf{M}_l$ appearing in the matrix form of $B$ with $\mathsf{D}_l$ and $h_l^{d}\mathsf{I}_{N_l}$, respectively; see~\cite[Proposition~4.9]{Xu:1992}.

We first compare $B_{\mathrm{Jac}}$ with $B$.
Let $\mathsf{M}_\tau$ be the local mass matrix on an element
$\tau\in\mathcal{T}_l$, and let
$\mathsf{D}_\tau=\operatorname{diag}(\mathsf{M}_\tau)$.
By the local mass matrix formula~\eqref{Lem1:trace_discrete}, we get
\begin{equation}
\label{mass_diagonal_local}
    \frac{1}{2}
    \mathsf{v}_\tau^{\mathsf T}\mathsf{D}_\tau\mathsf{v}_\tau
    \leq
    \mathsf{v}_\tau^{\mathsf T}\mathsf{M}_\tau\mathsf{v}_\tau
    \leq
    \frac{d+2}{2}
    \mathsf{v}_\tau^{\mathsf T}\mathsf{D}_\tau\mathsf{v}_\tau,
    \quad \mathsf{v}_\tau\in\mathbb{R}^{d+1}.
\end{equation}
Assembling~\eqref{mass_diagonal_local} over all elements gives
\[
    \frac{1}{2}
    \mathsf{v}^{\mathsf T}\mathsf{D}_l\mathsf{v}
    \leq
    \mathsf{v}^{\mathsf T}\mathsf{M}_l\mathsf{v}
    \leq
    \frac{d+2}{2}
    \mathsf{v}^{\mathsf T}\mathsf{D}_l\mathsf{v},
    \quad \mathsf{v}\in\mathbb{R}^{N_l}.
\]
This yields the following comparison.

\begin{proposition}
\label{Prop:BJac}
Suppose that \cref{Ass:triangulation,Ass:multilevel} hold.
Then the Bramble--Pasciak--Xu preconditioners $B$ and $B_{\mathrm{Jac}}$ defined in~\eqref{BPX} and~\eqref{BPX_variants} satisfy
\begin{equation*}
    \frac{1}{2}(Bv,v)_{L^2(\Omega)}
    \leq
    (B_{\mathrm{Jac}}v,v)_{L^2(\Omega)}
    \leq
    \frac{d+2}{2}(Bv,v)_{L^2(\Omega)},
    \quad v\in V_{h,0}.
\end{equation*}
Consequently, we have
\begin{equation*}
    \kappa(B_{\mathrm{Jac}}A)
    \leq
    (d+2)\kappa(BA).
\end{equation*}
\end{proposition}

Next, we compare $B_{\mathrm{Ric}}$ with $B_{\mathrm{Jac}}$.
Let $\omega_i^l$ be the vertex patch associated with the nodal basis function $\phi_i^l$.
We define
\[
    \eta_i^l
    =
    \frac{2|\omega_i^l|}{(d+1)(d+2)h_l^d},
    \quad
    1\leq i\leq N_l,\ 1\leq l\leq J,
\]
and set
\begin{equation}
\label{eta}
    \eta_-=\min_{\substack{1\leq l\leq J\\1\leq i\leq N_l}}\eta_i^l,
    \quad
    \eta_+=\max_{\substack{1\leq l\leq J\\1\leq i\leq N_l}}\eta_i^l.
\end{equation}
 Since
\[
    (\mathsf{D}_l)_{ii}
    =
    \|\phi_i^l\|_{L^2(\Omega)}^2
    =\eta_i^l h_l^d,
\]
the definitions of \(B_{\mathrm{Jac}}\) and \(B_{\mathrm{Ric}}\) give the following comparison.

\begin{proposition}
\label{Prop:BRic}
Suppose that \cref{Ass:triangulation,Ass:multilevel} hold.
Then the Bramble--Pasciak--Xu preconditioners $B_{\mathrm{Jac}}$ and $B_{\mathrm{Ric}}$ defined in~\eqref{BPX_variants} satisfy
\begin{equation*}
    \eta_- (B_{\mathrm{Jac}}v,v)_{L^2(\Omega)}
    \leq
    (B_{\mathrm{Ric}}v,v)_{L^2(\Omega)}
    \leq
    \eta_+ (B_{\mathrm{Jac}}v,v)_{L^2(\Omega)},
    \quad v\in V_{h,0},
\end{equation*}
where $\eta_-$ and $\eta_+$ are defined in~\eqref{eta}.
In particular,  if \(\eta_i^l\) is independent of \(i\) and \(l\), then
\(B_{\mathrm{Ric}}\) is a scalar multiple of \(B_{\mathrm{Jac}}\), and we have
\begin{equation*}
    \kappa(B_{\mathrm{Ric}}A)
    =
    \kappa(B_{\mathrm{Jac}}A)
    \leq
    (d+2)\kappa(BA),
\end{equation*}
where $B$ is defined in~\eqref{BPX}.
\end{proposition}

In summary, the Jacobi variant \(B_{\mathrm{Jac}}\) inherits the polynomial dimension dependence of \(B\), with only the additional factor \(d+2\) in the condition number estimate.
The Richardson variant \(B_{\mathrm{Ric}}\) is comparable to \(B_{\mathrm{Jac}}\) through the patch-volume constants \(\eta_-\) and
\(\eta_+\) defined in~\eqref{eta}.
In particular, for uniform multilevel triangulations with identical vertex patch volumes for all interior nodes at each level, these constants coincide,
and \(B_{\mathrm{Ric}}\) has the same condition number as \(B_{\mathrm{Jac}}\).
Thus, in this case, \(B_{\mathrm{Ric}}\) also preserves the polynomial dimension dependence.

\section{Conclusion}
\label{Sec:Conclusion}
In this paper, we have provided a comprehensive analysis of the dimension dependence of BPX preconditioners.
We proved that, under certain geometric assumptions, the BPX preconditioner exhibits only polynomial dependence on the spatial dimension. This is relevant to recent quantum algorithms proposed in~\cite{DP:2026,JLMY:2025} for solving elliptic PDEs, where the complexity of the quantum solvers depends on the conditioning of the preconditioned operator. We stress, however, that our estimates supply a necessary spectral component of the complexity story in high-dimensional quantum computing, while the algorithmic realization remains a separate issue.  

We note that one important assumption in this paper is the convexity of the domain, which guarantees the dimension-independent elliptic regularity estimate stated in \cref{Lem:regularity}.
Such a dimension-independent estimate is generally not available for nonconvex domains.
Since many practical applications involve nonconvex geometries, understanding the behavior of BPX preconditioners for high-dimensional problems without full elliptic regularity remains an important direction for future research.

Another natural direction is the analysis of BPX preconditioners for tensor-product \(\mathbb{Q}_1\) finite element spaces.
This setting is substantially different from the simplicial \(\mathbb{P}_1\) setting considered in the present paper.
In particular, for tensor-product multilinear finite elements, the stiffness matrix contains tensor-product mass-matrix factors, and its condition number may deteriorate exponentially with the spatial dimension; see, e.g.,~\cite[Remark~3.2]{Bachmayr:2023}.
Thus, a polynomial-in-\(d\) BPX analysis for tensor-product \(\mathbb{Q}_1\) elements would require additional ideas that address this tensor-product conditioning effect.
We therefore regard this case as a topic for future work.

\section*{Acknowledgment}
The authors wish to thank Professor Daniel Peterseim for highlighting the importance of BPX preconditioners in the context of quantum computing, which inspired this research, in his talk at the 21st European Finite Element Fair.
The authors are also grateful to Professor Shi Jin for his seminar during his visit to KAUST, and for posing the question that motivated this work.

\appendix

\section{A lower bound for the Scott--Zhang interpolation}
\label{App:SZ}

The estimates in \cref{Sec:SZ} use the averaged Scott--Zhang interpolation rather than the standard Scott--Zhang interpolation~\cite{SZ:1990} defined in \cref{Def:SZ}.
The purpose of this appendix is to illustrate, in a simple high-dimensional setting, why such a modification is useful when the dependence of constants on the dimension is taken into account. 
We show that, for the Freudenthal triangulation described in \cref{Sec:Triangulations}, the approximation
constant of the standard Scott--Zhang interpolation can grow super-exponentially with the dimension.

Let $\Omega=(0,1)^d$. We split each coordinate interval at $\frac{1}{2}$ and partition each of the resulting $2^d$ half-cubes by a conforming Freudenthal triangulation. We denote the resulting triangulation by $\mathcal{T}$ and the associated continuous piecewise linear finite element space by $V$. Let $I^{\mathrm{SZ}} \colon H^1(\Omega)\to V$ be any Scott--Zhang interpolation operator in the sense of \cref{Def:SZ}.
We show that
\begin{equation}
\label{SZ_exponential}
    \sup_{v\in H^1(\Omega),\ |v|_{H^1(\Omega)}\neq 0}
    \frac{\|v-I^{\mathrm{SZ}}v\|_{L^2(\Omega)}}
    {h|v|_{H^1(\Omega)}}
    \geq
    \frac{d!\sqrt{(3d)!}}
    {d(2d+1)!\sqrt{2^d(d+2)}}.
\end{equation}
Consequently, by the Stirling formula, we have
\begin{equation*}
    \sup_{v\in H^1(\Omega),\ |v|_{H^1(\Omega)}\neq 0}
    \frac{\|v-I^{\mathrm{SZ}}v\|_{L^2(\Omega)}}
    {h|v|_{H^1(\Omega)}}
    \geq  \exp \left( \frac{1}{2}d\log d - Cd \right)
\end{equation*}
for some constant \(C>0\) independent of \(d\).

The only interior node of \(\mathcal{T}\) is
\[
    a=\left(\frac{1}{2},\dots,\frac{1}{2}\right).
\]
Let \(K_a\) be the simplex selected by \(I^{\mathrm{SZ}}\) for the vertex \(a\), as in \cref{Def:SZ}.
Let \(\phi_a\) be the nodal basis function associated with \(a\), as defined in~\eqref{phi}.
By the same mass-matrix calculation as in the proof of \cref{Lem:psi}, the dual function associated with the vertex \(a\) on \(K_a\) is
\begin{equation*}
    \psi_a
    = \frac{d+1}{|K_a|} ((d+2)\lambda_0-1 ).
\end{equation*}

Let $\lambda_0,\ldots,\lambda_d$ be the barycentric coordinates on $K_a$, where $\lambda_0$ corresponds to the vertex $a$. Define
\begin{equation*}
    w(x)
    =
    \begin{cases}
        \prod_{j=0}^d \lambda_j(x), & x\in K_a, \\
        0, & x\in\Omega\setminus K_a .
    \end{cases}
\end{equation*}
Since $w$ vanishes on all faces of $K_a$, its extension by zero belongs to $H^1(\Omega)$.
It follows that
\begin{equation*}
    I^{\mathrm{SZ}}w = \alpha \phi_a,
\end{equation*}
where
\begin{multline*}
    \alpha
    =
    \int_{K_a}\psi_a w\,dx
    =
    \frac{d+1}{|K_a|}
    \left[
        (d+2)\int_{K_a}\lambda_0\prod_{j=0}^d\lambda_j\,dx
        -
        \int_{K_a}\prod_{j=0}^d\lambda_j\,dx
    \right] \\
    =
    (d+1)
    \left[
        (d+2)\frac{2d!}{(2d+2)!}
        -
        \frac{d!}{(2d+1)!}
    \right]
    =
    \frac{d!}{(2d+1)!}.
\end{multline*}
In the penultimate equality, we used the following barycentric-coordinate integration formula~\cite{VS:2018}:
\begin{equation*}
    \int_{K_a}\prod_{j=0}^d \lambda_j^{\beta_j}\,dx
    =
    \frac{d!\prod_{j=0}^d\beta_j!}
    {\left(d+\sum_{j=0}^d\beta_j\right)!}|K_a|,
\end{equation*}
where $\beta_0,\dots,\beta_d$ are nonnegative integers.

Since \(w\) vanishes outside \(K_a\), we obtain
\begin{multline}
\label{SZ_exponential_1}
    \|w-I^{\mathrm{SZ}}w\|_{L^2(\Omega)}^2
    \geq
    \alpha^2\|\phi_a\|_{L^2(\Omega\setminus K_a)}^2 \\
    =
    \alpha^2\sum_{\tau\subset\omega_a\setminus K_a}
    \int_\tau \phi_a^2\,dx
    =
    \alpha^2
    \frac{2((d+1)!-1)}{2^d(d+2)!}
    \geq
    \alpha^2\frac{1}{2^d(d+2)}.
\end{multline}
On the other hand, a direct calculation gives
\begin{equation}
\label{SZ_exponential_2}
|w|_{H^1(\Omega)}^2
    =
    \frac{4d}{(3d)!}.
\end{equation}
Combining~\eqref{SZ_exponential_1} and~\eqref{SZ_exponential_2}, and using \(h=\sqrt d/2\), yields~\eqref{SZ_exponential}.

\bibliographystyle{siamplain}
\bibliography{refs_BPX}

\end{document}